\documentclass[11pt, reqno]{amsart}
\usepackage{amsmath}
\usepackage{amssymb}
\usepackage{amsthm}
\usepackage{amsfonts}
\usepackage{bm}
\usepackage{braket}
\usepackage{mathrsfs,centernot,stmaryrd,color,tensor,wasysym}
\usepackage[truedimen,margin=25truemm]{geometry}
\usepackage{tcolorbox}
\usepackage{mathtools}
\mathtoolsset{showonlyrefs}

\usepackage{hyperref}

\allowdisplaybreaks

\theoremstyle{plain}
\newtheorem{th.}{Theorem}[section]
\newtheorem{theorem}{Theorem}[section]

\newtheorem{prop.}[th.]{Proposition}

\newtheorem{lem.}[th.]{Lemma}
\newtheorem{lemma}[th.]{Lemma}
\newtheorem{cor.}[th.]{Corollary}
\newtheorem{corollary}[th.]{Corollary}
\newtheorem{def.}[th.]{Definition}
\newtheorem{rmk.}[th.]{Remark}
\newtheorem{remark}[th.]{Remark}
\newtheorem{conj.}[th.]{Conjecture}

\makeatletter
\newcommand\thankssymb[1]{\textsuperscript{\@fnsymbol{#1}}}
\makeatother

\newcommand{\BC}{\mathbb{C}}
\newcommand{\BD}{\mathbb{D}}

\newcommand{\BR}{\mathbb{R}}
\newcommand{\BS}{\mathbb{S}}
\newcommand{\BT}{\mathbb{T}}

\newcommand{\vF}{{\mathbf{F}}}
\newcommand{\vG}{{\mathbf{G}}}

\newcommand{\vV}{{\mathbf{V}}}

\newcommand{\vX}{{\mathbf{X}}}

\newcommand{\CB}{\mathcal{B}}

\newcommand{\CD}{\mathcal{D}}

\newcommand{\CF}{\mathcal{F}}

\newcommand{\CH}{\mathcal{H}}

\newcommand{\CJ}{\mathcal{J}}

\newcommand{\CL}{\mathcal{L}}
\newcommand{\CM}{\mathcal{M}}
\newcommand{\CN}{\mathcal{N}}

\newcommand{\CR}{\mathcal{R}}
\newcommand{\CS}{\mathcal{S}}

\newcommand{\CU}{\mathcal{U}}

\newcommand{\CW}{\mathcal{W}}

\newcommand{\FS}{\mathfrak{S}}

\newcommand{\al}{\alpha}

\newcommand{\ga}{\gamma}
\newcommand{\de}{\delta}
\newcommand{\ep}{\varepsilon}

\newcommand{\ph}{\varphi}
\newcommand{\ta}{\tau}

\renewcommand{\th}{\theta}
\newcommand{\si}{\sigma}

\newcommand{\rh}{\rho}
\newcommand{\ch}{\chi}

\newcommand{\De}{\Delta}
\newcommand{\Ph}{\Phi}

\newcommand{\pl}{\partial}
\newcommand{\wt}{\widetilde}
\newcommand{\wh}{\widehat}

\newcommand{\Dk}[1]{\left[#1\right]}
\newcommand{\K}[1]{\left(#1\right)}
\newcommand{\Br}[1]{\langle #1 \rangle}

\newcommand{\abs}[1]{\left|#1\right|}
\newcommand{\No}[1]{\left\| #1 \right\|}

\def\bra#1{\langle#1\rangle}

\newcommand{\I}{\infty}

\newcommand{\Tr}{{\rm Tr}}

\newcommand{\sd}{\langle \nabla \rangle}

\newcommand{\sdi}{\sd^{-1}}

\newcommand{\lxr}{\langle \xi \rangle}

\newcommand{\tx}{{t,x}}
\newcommand{\Sa}{{\FS^\al}}

\newcommand{\Sxtx}{{\FS^{2\al'}_{x \to (t,x)}}}
\newcommand{\Stxx}{{\FS^{2\al'}_{(t,x) \to x}}}
\newcommand{\st}{\star}

\newcommand{\ft}{{\frac{1}{4}}}
\newcommand{\fh}{\frac{3}{4}}

\newcommand{\R}{\mathbb{R}}

\newcommand{\tw}{\frac{1}{2}}
\newcommand{\na}{\nabla}
\newcommand{\oc}{\ocircle}
\newcommand{\os}{\oast}
\newcommand{\ls}{\lesssim}

\title[Stationary solutions of the Hartree and Schr\"{o}dinger equations]{Asymptotic stability of a wide class of stationary solutions for the Hartree and Schr\"{o}dinger equations for infinitely many particles}
\author{Sonae Hadama\thankssymb{2}}
\thanks{\thankssymb{2} Research Institute for Mathematical Sciences, Kyoto University, Kita-Shirakawa, Sakyo-ku, Kyoto, Japan 606-8502.
E-mail address: \texttt{hadama@kurims.kyoto-u.ac.jp}}

\begin{document}
\maketitle

\begin{abstract}
We consider the Hartree and Schr\"{o}dinger equations describing the time evolution of wave functions of infinitely many interacting fermions in three-dimensional space.
These equations can be formulated using density operators, and they have infinitely many stationary solutions.
In this paper, we prove the asymptotic stability of a wide class of stationary solutions.
We emphasize that our result includes Fermi gas at zero temperature.
This is one of the most important steady states from the physics point of view; however, its asymptotic stability has been left open after the seminal work of Lewin and Sabin \cite{2014LS}, which first formulated this stability problem and gave significant results.
\end{abstract}

\vspace{+4mm}



\tableofcontents

\section{Introduction}
In this paper, we study the following Hartree equation in three-dimensional space:
\begin{align} \tag{NLH} \label{NLH}
 i \pl_t \ga &= [- \Delta + w \ast \rho_\ga, \ga], \quad \ga:\BR \to \CB(L^2_x).
\end{align}      
This is a nonlinear evolution equation for operator-valued functions.
We denote the set of all bounded linear operators on $L^2_x := L^2(\BR^3)$ by $\CB(L^2_x)$,
convolution in space by $\ast$ and commutator by $[\cdot, \cdot]$.
We assume that $w$ is a given finite signed Borel measure on $\BR^3$.
For $A \in \CB(L^2_x)$, $\rh_A(x) : \BR^3_x \to \BC$ is the density function of $A$, 
that is, $\rh_A(x):=k(x,x)$, where $k(x,y)$ is the integral kernel of $A$.
When $w = \pm \de$, we should call \eqref{NLH} the cubic nonlinear Schr\"{o}dinger equation;
however, we consistently call \eqref{NLH} the Hartree equation in the rest of this paper. 

We briefly explain the background and fundamental properties of \eqref{NLH} for convenience according to \cite{2015LS}. 
See \cite[Introduction]{2015LS} for more details.

Wave functions of $N$ fermions in three-dimensional space evolve according to a linear Schr\"{o}dinger equation on $\BR^{3N}$.
The Hartree equation is one of the approximations of this equation:
\begin{align} \label{originalH}
            i \pl_t u_n(t,x) = \left(- \Delta_x +  w \ast \left( \sum_{j=1}^N |u_j(t,x)|^2 \right) \right) u_n(t,x),
              \quad (t,x) \in \BR \times \BR^3, \quad (n=1, \dots N).
\end{align}
If we try to deal with the case $N = \I$, namely, the Hartree equation for infinitely many particles,
it seems somewhat difficult to do our analysis in this form.
Thus, it is a good idea to rewrite the original equation \eqref{originalH}.
For a solution to \eqref{originalH}, $(u_n(t))_{n=1}^N \subset C(\BR, L^2_x)$,
define the operator-valued function $\ga : \BR \to \CB(L^2_x)$ by
\begin{equation*}
    \ga(t) = \sum_{n=1}^N |u_n(t)\rangle  \langle u_n(t)|,
\end{equation*}
where $|f \rangle \langle g | : \phi \mapsto \Br{g | \phi}_{L^2_x} f$.
It is quite easy to see that $\ga(t)$ solves \eqref{NLH}.
\eqref{NLH} is more suitable for our analysis because it is not dependent on the number of particles $N$.

The following proposition is easy to prove, but it is one of the most essential properties of \eqref{NLH}.
\begin{prop.}\label{SW}
 Let $w$ be a finite Borel measure on $\BR^3$.
 Let $f \in L^1_\xi \cap L^\I_\xi$.
 Let $\ga_f := f(- i \nabla)= \CF^{-1} f \CF$.
 Then $\ga_f$ is a stationary solution to \eqref{NLH}.
\end{prop.}
By Proposition \ref{SW}, we have infinitely many stationary solutions to \eqref{NLH}.
Furthermore, it is known that some $f$ are important from the physics perspective:
\begin{align}
  &f(\xi) = \ch_{\{|\xi|^2 \leq \mu\}}(\xi),\quad  \mbox{(Fermi gas at zero temperature and $\mu >0$)}, \label{P1} \\
  &f(\xi) = \frac{1}{e^{(|\xi|^2-\mu)/T}+1}, \quad \mbox{(Fermi gas at positive temperature $T$ and $\mu \in \R$)},  \label{P2} \\
  &f(\xi) = \frac{1}{e^{(|\xi|^2-\mu)/T}-1}, \quad \mbox{(Bose gas at positive temperature $T$ and $\mu < 0$)}, \label{P3} \\
  &f(\xi) = e^{-(|\xi|^2-\mu)/T}, \quad \mbox{(Boltzmann gas at positive temperature $T$ and $\mu \in \R$)},  \label{P4}
\end{align}
where $\ch_A$ is the indicator function of $A$ and each $\mu$ is chemical potential.
Proposition \ref{SW} tells us that \eqref{NLH} with non-trace class initial values is important since $\ga_f$ is not compact unless $f=0$.

In this paper, we study the asymptotic stability of $\ga_f$.
Let $Q(t)$ be a perturbation from $\ga_f$, that is, $Q(t) = \ga(t) - \ga_f$.
Then we have
\begin{align} \tag{$f$-NLH} \label{$f$-NLH}
i \pl_t Q = [-\De + w \ast \rh_Q, \ga_f + Q].
\end{align}

We denote the Schatten $\al$-class on a Hilbert space $H$ by $\FS^\al(H)$, that is, 
\begin{align}
	\|A\|_{\FS^\al}:= \K{\Tr(|A|^\al) }^{\frac{1}{\al}},
\end{align}
where $\Tr$ is the trace in $H$.
We write $\FS^\al:=\FS^\al(L^2_x)$.
We refer to \cite{2005Simon} for more details on the Schatten classes.
Note that $\Tr A$ is the total number of particles contained in the system for any density operator $A$,
and $\Tr A = \I$ at least formally if $A \in \FS^\al \setminus \FS^1$ for some $\al >1$.

The well-posedness of the Cauchy problem of \eqref{NLH} with $\ga(0)=\ga_0$ has been studied by \cite{1974BDF, 1976BDF, 1976C, 1992Z}, 
and recently, the small data scattering is shown in \cite{2021PS}.
More precisely, the above results dealt with more general nonlinear terms.
However, they considered only trace class operators; their arguments are in the framework of finite particle systems.
We need to emphasize that Lewin and Sabin \textit{first} formulated \eqref{$f$-NLH} and gave significant results in \cite{2015LS, 2014LS}.
In \cite{2015LS}, they proved the local and global well-posedness of \eqref{$f$-NLH} with non-trace class initial data when the interaction potential $w$ is in $L^1_x \cap L^\I_x$.
In \cite{2014LS}, they showed asymptotic stability of $\ga_f$ for small initial perturbations $Q_0$ in two-dimensional space.
They assumed that interaction potential $w$ is in the Sobolev space $W^{1,1}(\BR^2)$ and $f$ is smooth.
Chen, Hong and Pavlovi\'{c} extended the above results.
In \cite{2017CHP}, they proved the global well-posedness when the potential $w$ is the delta measure and $f$ is Fermi gas at zero temperature.
In \cite{2018CHP}, they extended the result in \cite{2014LS} and proved scattering when $d \ge 3$.
However, they also assumed that $f$ is smooth and $w$ satisfies somewhat complicated conditions.
The Hartree equation with a constant magnetic field also has infinitely many stationary solutions.
In \cite{2021D}, Dong gave a well-posedness result of the Cauchy problem that initial data are given around these stationary solutions.

Strichartz estimates for orthonormal functions is one of the most essential tools for their analysis.
This is first proven in \cite{2014FLLS}, and developed in \cite{2017FS, 2019BHLNS}.

In relation to \cite{2015LS, 2014LS}, Lewin and Sabin gave rigorous proof that the semi-classical limit of the Hartree equation for infinitely many particles is the Vlasov equation in \cite{2020LS}.
And they get the global well-posed result for the Vlasov equation as a by-product.

We rewrote \eqref{originalH} in terms of density operators on $L^2_x$,
and all of the known results mentioned above are based on this formulation; 
however, de Suzzoni gave an alternative formulation by random fields in \cite{2015S}.
She clarified the correspondence with the formulation using density operators in detail:
For example, we have infinitely many steady states corresponding to $\ga_f$ in random fields formulation.
In \cite{2020CS}, Collot and de Suzzoni proved their asymptotic stability when $d\geq 4$, and they extended their result to $d=2,3$ in \cite{2022CS}.
Remarkably, the general finite Borel measure on $\R^3$ is allowed in \cite{2022CS}.

In the previous works, the \textit{asymptotic} stability of $\ga_f$ given by \eqref{P2}, \eqref{P3} and \eqref{P4}, which are physically important examples at positive temperature, was proved.
However, the asymptotic stability of Fermi gas at zero temperature \eqref{P1} has been left open since it was mentioned as an open question in the seminal work \cite{2014LS}.
Moreover, no scattering result allows singular interaction potentials at the density operator level.
In this paper, we prove the asymptotic stability of $\ga_f$ when $f$ is in a wide class including $\ch_{\{|\xi|^2\le 1\}}$, allowing general finite Boreal measure on $\BR^3$.

\subsection{Main result}
We write $U(t):=e^{it\De}$.
Define $A_\st B := A B A^*$ for two operators $A$ and $B$.
For $s \ge 0$ and $\al \in [1,\I]$, we define the Schatten--Sobolev space $\CH^{s,\al}$ by
\begin{align}
 \|A\|_{\CH^{s,\al}} := \|\sd^s A \sd^s\|_\Sa.
\end{align}
We write $\CH^{s}:= \CH^{s,2}$.
For $w$ and $f$, we define a linear operator $\CL_1 = \CL_1(f,w)$ by
\begin{align}
	\CL_1[g](t) := \rh \K{i\int_0^t U(t-\ta)_\st [w \ast g(\ta), \ga_f] d\ta}.
\end{align}

Our main result is the following:
\begin{th.} \label{main}
	Let $w$ be a given finite signed Borel measure on $\BR^3$.
	Let $f \in L^1_\xi \cap L^\I_\xi$ satisfy $\|\lxr^2 f\|_{L^1_\xi \cap L^\I_\xi} < \I$.
	Assume that the operator $\CL_1 = \CL_1(f,w)$ satisfies
	\begin{align}
		\|(1+\CL_1)^{-1}\|_{\CB(L^2_{t,x})} < \I.
	\end{align}
	Then there exists $\ep_0 > 0$ such that the following holds:
	If $\| Q_0 \|_{\CH^{1/2,3/2}} \leq \ep_0$,
	then there exists a unique global solution $Q(t) \in C(\BR, \CH^\tw)$ to \eqref{$f$-NLH} with $Q(0)=Q_0$
	such that $\rh_Q \in L^2_t(\BR, H^\tw_x)$.
	Furthermore, $Q(t)$ scatters; that is,
	there exist $Q_{\pm} \in \FS^3$ such that
	\begin{align}
		U(-t)Q(t)U(t) \to Q_\pm \mbox{ in }  \FS^3 \mbox{ as } t \to \pm \I.
	\end{align}
\end{th.}

First, we discuss the invertibility of $1+\CL_1$.
Some sufficient conditions for $(1+\CL_1)^{-1} \in \CB(L^2_{t,x})$ are known.
For example, see \cite[Corollary 1]{2014LS}.
Although the author could not find any clear statement, \cite[Proposition 1]{2014LS} immediately implies the following:
\begin{prop.}
	Let $w$ be a finite signed Borel measure on $\R^3$. 
	Let $f \in L^1_\xi \cap L^\I_\xi$ be real-valued and radial.
	If
	\begin{align}
		\frac{\|\widehat{w}\|_{L^\I_\xi}}{2|\BS^{d-1}|} \int_{\BR^d} \frac{|\check{f}(x)|}{|x|^{d-2}} dx < 1
		\quad \mbox{or}
		\quad \No{\frac{\wh{w}(\xi)}{|\xi|}}_{L^\I_\xi} \int_0^\I |\check{f}(r)|dr < 1
	\end{align}
	holds, then $(1+\CL_1(f,w))^{-1} \in \CB(L^2_{t,x})$.
\end{prop.}
\begin{rmk.}
	We identify a radial function $H:\BR^d \to \BR$ and $h:[0,\I) \to \BR$ such that $H(x) = h(|x|)$.
\end{rmk.}

Our result includes the Fermi gas at zero temperature because we have the following criteria:
\begin{prop.}[\cite{2023H} Proposition 1.6] \label{cor}
	Let $w$ be a finite signed Borel measure on $\BR^3$.
	Let $f(\xi) = \ch_{\{|\xi|^2 \le 1\}}$.
	If $\wh{w}$ is real-valued and
	\begin{align} \label{potential 3D}
		-\de_0 \le \wh{w}(\xi) \le \frac{\de_1}{\Br{\log|\xi|}}
	\end{align}
	for (small) absolute constants $\de_0, \de_1 > 0$,
	then $(1+\CL_1(f,w))^{-1} \in \CB(L^2_t L^2_x)$.
\end{prop.}
	
	Now we give some remarks about our main results.

\begin{rmk.}
When $f(\xi) = \ch_{\{|\xi|^2 \le 1\}}$, the scattering holds only in the focusing case by \eqref{potential 3D}.
\end{rmk.}

\begin{rmk.}
	If we include the delta measure as an interaction potential of \eqref{NLH},
	Theorem \ref{main} is optimal, meaning cubic NLS in 3D is critical in $H^\tw_x$. 
	However, it may be an interesting problem to improve the scattering norm or the function space that the solution $Q(t)$ belongs to at each time.
	In particular, the author suspects that it may be possible to replace $\FS^3$ in Theorem \ref{main} by $\CH^{\tw,3}$.
\end{rmk.}

\subsection{Summary of ideas of the proof of the main result}
The rough story of the proof is the same as that of \cite{2014LS} and \cite{2018CHP}.
First, we reduce the Cauchy problem \eqref{$f$-NLH} with $Q(0)=Q_0$ to the nonlinear equation of the density function $\rh_Q$ (see \eqref{IVP*}).
After finding a unique global solution $\rh_Q$, we restore the scattering solution $Q(t) \in C(\BR,\CH^\tw)$.
See Sect.2 for more details.
Most of our efforts are devoted to finding a unique global density function $\rh_Q$ in $L^2_t(\R,H^\tw_x)$.

One of the most essential tools in this paper is the Strichartz estimates for orthonormal functions; in particular, we essentially use the estimates proved by Bez, Hong, Lee, Nakamura and Sawano in \cite{2019BHLNS}.
This is the estimate for the free propagator $e^{it\De}$, but we need to extend their result to the general propagator $U_V(t)$ (see \eqref{LS} and \eqref{IVP*}).
To do so, according to \cite{2014LS, 2018CHP}, we use the wave operator decomposition of $U_V(t)$ (see \eqref{UVWV}).
We need to emphasize that \cite{2019BHLNS} appeared after the important previous works \cite{2014LS, 2018CHP} were written, and the estimates established in \cite{2019BHLNS} enabled us to get the result in this paper.

Another important ingredients in this paper is Christ--Kiselev type lemma in Schatten classes (see Lemma \ref{SCK}).
It is a direct corollary of the result in \cite{1970Go-Kre},
and it was obtained much earlier than Christ--Kiselev's result \cite{2001CK};
therefore, we should call it Gohberg--Kre\u{\i}n theorem.
Gohberg--Kre\u{\i}n theorem and the duality principle, first introduced in \cite{2017FS} to prove the orthonormal Strichartz estimates, is also very useful. 
The author would like to emphasize that both the duality principle and Gohberg--Kre\u{\i}n theorem are known results, but it was not noticed that combining them is useful as far as the author knows.

It is natural to find a density function $\rh_Q$ such that $w \ast \rh_Q \in L^2_t(\BR,H^\tw_x)$.
However, we would like to deal with general finite signed Borel measures on $\R^3$; hence, we need to find $\rh_Q$ in $L^2_t(\R,H^\tw_x)$.
It means that we need to estimate $\|\rh(U_V(t)_\st Q_0)\|_{L^2_t H^{1/2}_x}$ (see \eqref{IVP*}).
One of the most difficult part of this paper is to bound this term, and the difficulty comes from $H^\tw_x$.
If we replace $H^\tw_x$ by $H^n_x$ with integer $n$, then we can bound it much easier because we have the formula
$\na \rh_Q = \rh([\na,Q])$, 
and we can easily calculate $[\na,U_V(t)]$ explicitly; however, for the fractional derivatives, there does not exist this type of calculation as far as the author knows.
Hence, we ``interpolate'' the bounds of $\|\rh(U_V(t)_\st Q_0)\|_{L^2_t L^2_x}$ and $\|\rh(U_V(t)_\st Q_0)\|_{L^2_t H^1_x}$, but things do not go so straightforwardly because the bounds of these terms are nonlinear with respect to $V$.
To overcome this problem, we decompose $U_V(t)$ into wave operators and then multilinearize them.
In this argument, Gohberg--Kre\u{\i}n theorem (and duality argument) plays an essential role.

Finally, we will estimate the terms including $\ga_f$. 
In the first version of this paper, the author said ``we can bound them by simple duality arguments''.
However, it is not true. 
We will explain this difficulty here.
We often need to get estimate in the following form
\begin{align}\label{eq:retarded}
	\No{\int_0^\I dt U(t)^* (\sd^\tw V)(t)U(t) \int_0^t ds U(s)^*W(s)U(s) }_{\FS^\al} \ls \|V\|_{L^2_t L^p_x} \|\sd^\tw W\|_{L^2_t L^q_x}.
\end{align}
In the most cases, it is pretty easy to get
\begin{align}\label{eq:full int}
		\No{\int_0^\I dt U(t)^* (\sd^\tw V)(t)U(t) \int_0^\I ds U(s)^*W(s)U(s) }_{\FS^\al} \ls \|V\|_{L^2_t L^p_x} \|\sd^\tw W\|_{L^2_t L^q_x}.
\end{align}
However, the author do not know any way to conclude \eqref{eq:retarded} from \eqref{eq:full int} directly. 
To overcome this problem, we first replace $\sd^\tw$ to $\sd^j$ with $j=0,1$, and then get similar estimates.
Lastly, we interpolate them. See Sect.5 for detailed arguments.

\subsection{Organization of this paper}
This paper is organized as follows.
In Sect.2, we explain the outline of the proof of the main result.
We reduce the Cauchy problem \eqref{NLH} with $Q(0)=Q_0$ to the equation of density function $\rh_Q$.
This is the density functional method.
In Sect.3, we prepare a lot of basic tools.
It includes Christ-Kiselev type lemma, resolvent expansion of the propagator $U_V$, and so on.
In Sect.4, we prove the key Strichartz estimates in this paper.
In Sect.5, we show the global and unique existence of the density function $\rh_Q$.
Finally, in Sect.6, we restore the solution $Q(t)$ to the original Cauchy problem from its density function $\rh_Q$.

\section{Outline of the proof of the main result}
\subsection{Reduction to the equation for density functions}
Let $U_V(t,s)$ be the propagator for the following linear Schr\"{o}dinger equation with a time-dependent potential:
 \begin{align} \label{LS}
  (i \pl_t + \De -V(t,x))u = 0, \quad u(t,x) : \BR^{1+3} \to \BC, \quad V(t,x):\BR^{1+3} \to \BR,
 \end{align}
namely, for any $s \in \BR$, $u(t):= U_V(t,s)u(s)$ is a solution to \eqref{LS}.

Our strategy is the same as that of \cite{2014LS} and \cite{2018CHP}.
First, we find a solution to the nonlinear equation of density function $\rh_Q$:
\begin{align}\tag{IVP} \label{IVP}
 &\rh_Q = \rh(U_V(t)_\st Q_0) - \rh\Dk{i\int_0^t U_V(t,\ta)_\st [V(\ta),\ga_f] d\ta}, \\
 &V=w \ast \rh_Q.
\end{align}
If we find a solution $\rh_Q$ to \eqref{IVP},
then we can restore the solution to \eqref{NLH} with $Q(0)=Q_0$ by
\begin{align} \label{eq Q}
 &Q(t) := U_V(t)_\st Q_0 - i\int_0^t U_V(t,\ta)_\st [V(\ta),\ga_f] d\ta, \\
 &V := w \ast \rh_Q.
\end{align}
We can reduce the proof of our main result to finding the solution $\rh_Q$ to \eqref{IVP} because we have the following lemma:
\begin{lem.} \label{reduction}
Let $d=3$.
Let $\Br{\xi}^2 f(\xi) \in L^1_\xi \cap L^\I_\xi$, $Q_0 \in \CH^\tw$ and $V \in L^2_t(\R,H^\tw_x)$.
Then $Q(t)$ defined by \eqref{eq Q} is in $C(\BR, \CH^\tw)$.
Furthermore, $Q(t)$ scatters; that is,
there exist $Q_\pm \in \FS^3$ such that
\begin{align}
 U(-t)Q(t)U(t) \to Q_\pm \mbox{ in } \FS^3 \mbox{ as } t \to \pm \I.
\end{align}
\end{lem.}

\subsection{Set-up for the contraction mapping argument}
We transform \eqref{IVP} to a more suitable form to estimate.
Note that $U_V(t,s)$ satisfies
\begin{align}
 U_V(t,s) &= U(t-s) -i \int_s^t U(t-\ta) V(\ta) U_V(\ta,s) d\ta  \\
 &=:U(t-s) + D_V(t,s).
\end{align}
Hence, we have
\begin{align}
 &-i\int_0^t U_V(t,\ta)_\st [V(\ta),\ga_f] d\ta \\
 &\quad = -i \int_0^t (U(t-\ta)+D_V(t,\ta))[V(\ta), \ga_f] (U(\ta-t)+D_V(\ta,t))d\ta \\
 &\quad = -i \int_0^t U(t-\ta)[V(\ta), \ga_f] U(\ta-t)d\ta 
 -i \int_0^t U(t-\ta)[V(\ta), \ga_f] D_V(\ta,t)d\ta \\
 &\qquad -i \int_0^t D_V(t,\ta)[V(\ta), \ga_f] U(\ta-t)d\ta 
 -i \int_0^t D_V(t,\ta)[V(\ta), \ga_f] D_V(\ta,t)d\ta \\
 &\quad =: L_1[V](t) + N_1[V](t) + N_2[V](t) + N_3[V](t). \label{ABCD}
\end{align}
We rewrite \eqref{IVP} and obtain
\begin{align}
 \rh_Q &= \rh(U_V(t)_\st Q_0) + \rh(L_1[V])+ \rh(N_1[V])+ \rh(N_2[V])+ \rh(N_3[V]) \\
 &= \rh(U_V(t)_\st Q_0) -\CL_1[\rh_Q] + \rh(N_1[V])+ \rh(N_2[V])+ \rh(N_3[V]).
\end{align}
Since we assumed $(1+\CL_1)^{-1}\in \CB(L^2_{t,x})$, we get
\begin{align} \label{IVP*} \tag{IVP*}
 \rh_Q = (1+\CL_1)^{-1}(\rh(U_V(t)_\st Q_0)  +\rh(N_1[V])+ \rh(N_2[V])+ \rh(N_3[V])).
\end{align}
Combining Theorem \ref{reduction}, we can reduce the proof of the main result to the following theorem:
\begin{th.} \label{den main}
	Under the same assumptions as in Theorem \ref{main}, there exists a unique global solution $\rh_Q \in L^2_t(\BR,H^\tw_x)$ to \eqref{IVP*}.
\end{th.}

\section{Preliminaries}
In this section, we consider general $d$-dimensional spaces because there is no benefit to limiting our argument to three-dimensional space.

\subsection{Christ--Kiselev type lemma}
First, we give a Chirst--Kiselev type lemma in Schatten classes;
however, the following result is already obtained in \cite{1970Go-Kre} before \cite{2001CK}.
Therefore, we should call it Gohberg--Kre\u{\i}n theorem.
We obtain the following result by applying \cite[Theorem 7.2]{2003BS} (but originally by \cite[Theorem III.6.2]{1970Go-Kre}),
setting $E_t = F_t = P_{L^2((a,t);L^2_x)}$ in their notations, where $P_A$ is a projection to $A$.
Note that the integral kernel of $\BD = \CJ^{E,F}_\th \BT$ is $\th(t-\ta)K(t,\ta)$, where $\th(x)=\ch_{\{x \ge 0\}}$.

\begin{theorem}[Gohberg--Kre\u{\i}n] \label{SCK}
	Let $d \ge 1$.
	Let $-\I \leq a < b \leq \I$, $1 < p < \I$ and $\al \in (1,\I)$.
	Let $(a,b)^2 \ni (t,\ta) \mapsto K(t,\ta)\in \CB(L^2_x)$ be strongly continuous.
	Assume that we can write $\BT \in \CB(L^2_\tx)$ by
	\begin{align}
		(\BT g)(t,x) = \int_a^b K(t,\ta)g(\ta)d\ta.
	\end{align}
	If $\BT \in \FS^\al(L^2_\tx)$, then  $\BD$ defined by
	\begin{align}
		(\BD g)(t,x) = \int_a^t K(t,\ta)g(\ta)d\ta
	\end{align}
	is also in $\FS^\al(L^2_\tx)$, and 
	\begin{align}
		\|\BD\|_{\FS^\al(L^2_\tx)} \le C_p \|\BT\|_{\FS^\al(L^2_\tx)}
	\end{align}
	holds.
\end{theorem}

\subsection{Duality principle and its application}
In this subsection, we assume all Hilbert spaces are separable and infinite-dimensional.
Let $H, K$ be complex Hilbert spaces.
Let $\al \in [1,\I]$ and $A: H \to K$ be compact.
We define Schatten $\al$-norm by
\begin{align}
	\|A\|_{\FS^\al(H\to K)}
	:= 
	\begin{cases}
		&\K{\Tr_H(|A^*A|^{\frac{\al}{2}})}^{\frac{1}{\al}} \mbox{ if } 1 \le \al <\I,\\
		&\|A\|_{\CB(H \to K)} \mbox{ if } \al = \I.
	\end{cases}
\end{align}
Note that $\FS^\al(H \to H)=\FS^\al(H)$.
H\"{o}lder's inequality for Schatten norm is well-known: If $\al, \al_0, \al_1 \in [0,\I]$ satisfy $\frac{1}{\al}=\frac{1}{\al_0}+\frac{1}{\al_1}$, then it holds that
\begin{align} \label{Holder}
	\|A_0 A_1\|_{\FS^{\al}(H)} \le \|A_0\|_{\FS^{\al_0}(H)} \|A_1\|_{\FS^{\al_1}(H)}. 
\end{align}
Note that \eqref{Holder} implies the following:
If $\al, \al_0, \al_1 \in [0,\I]$ satisfy $\frac{1}{\al}=\frac{1}{\al_0}+\frac{1}{\al_1}$, then it follows that
\begin{align} \label{g Holder}
	\|A_0 A_1\|_{\FS^{\al}(H_0\to H_2)} \le \|A_0\|_{\FS^{\al_0}(H_1 \to H_2)} \|A_1\|_{\FS^{\al_1}(H_0 \to H_1)}. 
\end{align}
In fact, by the singular value decomposition, we can write $A_0 = U \wt{A_0}$, where $\|\wt{A_0}\|_{\FS^{\al_0}(H_1)} =  \|A_0\|_{\FS^{\al_0}(H_1 \to H_2)}$ and $U:H_1 \to H_2$ is unitary.
In the same way, $A_1 = \wt{A_1} V$, where $\|\wt{A_1}\|_{\FS^{\al_1}(H_1)} = \|A_1\|_{\FS^{\al_1}(H_0 \to H_1)}$ and $V:H_0 \to H_1$ is unitary.
Therefore, we obtain
\begin{align}
	\|A_0 A_1\|_{\FS^\al(H_0 \to H_2)}
	&= \|U \wt{A_0} \wt{A_1} V\|_{\FS^\al(H_0 \to H_2)} 
	= \|\wt{A_0} \wt{A_1}\|_{\FS^\al(H_1)} \\
	&\le \|\wt{A_0}\|_{\FS^{\al_0}(H_1)} \|\wt{A_1}\|_{\FS^{\al_1}(H_1)}
	= \|A\|_{\FS^{\al_0}(H_1 \to H_2)} \|A_1\|_{\FS^{\al_1}(H_0 \to H_1)}.
\end{align}

For any $\al \in [1, \I]$, define
\begin{align}
	&\FS^\alpha := \FS^\alpha(L^2_x \to L^2_x),
	\quad \FS^\al_{t,x} := \FS^\al(L^2_{t,x} \to L^2_\tx), \\
	&\FS^\al_{x \to (t,x)} := \FS^\al(L^2_x \to L^2_{t,x}),
	\quad \FS^\alpha_{(t,x) \to x} := \FS^\al(L^2_{t,x} \to L^2_x).
\end{align}
Let $I \subset \R$ be an interval.
Let $A(t) \in \CB(L^2_x)$ for all $t \in I$ and $\sup_{t \in I} \|A(t)\|_{\CB} < \I$.
Assume that $I \ni t \mapsto A(t) \in \CB(L^2_x)$ be strongly continuous.
We define $A^\ocircle : L^2_x \to C_t(I,L^2_x)$ and $A^\oast : L^1_t(I,L^2_x) \to L^2_x$ by
\begin{align}
	&(A^\ocircle u_0)(t,x) := (A(t)u_0)(x), \\
	&(A^\oast f)(x) := \int_I A(\ta)^* f(\ta,x) d\ta.
\end{align}
Note that $A^\ocircle$ and $A^\oast$ are formally adjoint to each other.
The following lemma is very important for our analysis, but it is essentially the same thing as \cite[Lemma 3]{2017FS}.
\begin{lemma}[Duality principle] \label{DP}
	Let $p,q,\al \in [1,\I]$.
	The following are equivalent:
	
	(1) For any $\ga \in \FS^\al$, 
	\begin{align} \label{SS}
		\|\rh(A(t)_\st \ga)\|_{L^p_t(I,L^q_x)} \le C \|\ga\|_{\FS^\al}.
	\end{align}
	
	(2) For any $f \in L^{2p'}_t(I,L^{2q'}_x)$, 
	\begin{align} \label{HSS}
		\|f A^\oc\|_{\FS^{2\al'}_{x \to (\tx)}} \le C'\|f\|_{L^{2p'}_t(I,L^{2q'}_x)}.
	\end{align}
	
	(3) For any $f \in L^{2p'}_t(I,L^{2q'}_x)$, 
	\begin{align} \label{HSS}
		\|A^\os f\|_{\FS^{2\al'}_{(\tx) \to x}} \le C''\|f\|_{L^{2p'}_t(I,L^{2q'}_x)}.
	\end{align}
	Moreover, $\sqrt{C}$, $C'$ and $C''$ coincide.
\end{lemma}

\begin{proof}
	Assume (1). Then we have
	\begin{align}
		\|f A^\oc\|_\Sxtx^2 &= \|A^\os |f|^2 A^\oc\|_{\FS^{\al'}} \\
		&= \sup \{ |\Tr[\ga A^\os |f|^2 A^\oc ]| : \|\ga\|_{\FS^\al} \le 1\} \\
		&= \sup \left \{ \abs{\int_I \Tr[A(t) \ga A(t)^* |f(t)|^2]dt  } : \|\ga\|_{\FS^\al} \le 1 \right \} \\
		&\le \sup \left \{ \|\rh(A(t)_\st \ga)\|_{L^p_t(I,L^q_x)} \|f\|_{L^{2p'}_t(I,L^{2q'}_x)}^2  : \|\ga\|_{\FS^\al} \le 1 \right \} \\
		&\le C \|f\|_{L^{2p'}_t(I,L^{2q'}_x)}^2.
	\end{align}
	Therefore, we obtain 
	\begin{align}
		\|f A^\oc\|_\Sxtx \le \sqrt{C} \|f\|_{L^{2p'}_t(I,L^{2q'}_x)}.
	\end{align}
	We can prove that (2) yields $C \le (C')^2$ similarly.
	Since
	\begin{align}
		\|f A^\oc\|_{\Sxtx} = \|A^\os f\|_{\Stxx},
	\end{align}
	we completed the proof.
\end{proof}

The following lemma is sometimes useful:
\begin{lemma} \label{asym}
	Let $I \subset \R$ be an interval.
	Assume that
	\begin{align}
		\|\rh(A_n(t)_\st \ga)\|_{L^{p_n}_t(I,L^{q_n}_x)} \le C_n \|\ga\|_{\FS^{\al_n}}
	\end{align}
	for $n=0,1$.
	Then we have
	\begin{align}
		\|\rh(A_0(t) \ga A_1(t)^*)\|_{L^p_t(I,L^q_x)} \le \sqrt{C_0 C_1} \|\ga\|_{\FS^\al},
	\end{align}
	where
	\begin{align}
		\frac{1}{X} = \frac{1}{2} \K{\frac{1}{X_0}+ \frac{1}{X_1}}, \quad X = p,q,\al. 
	\end{align}
\end{lemma}

\begin{proof}
	Let $f \in L^{p'}_t(I,L^{q'}_x)$. 
	$f$ can be decomposed as $f = f_0 f_1$,
	where
	\begin{align}
		\|f\|_{L^{p'}_t(I,L^{q'}_x)} = \|f_0\|_{L^{2p_0'}_t(I,L^{2q_0'}_x)} \|f_1\|_{L^{2p_1'}_t(I,L^{2q_1'}_x)}.
	\end{align}
	Then we have
	\begin{align}
		\abs{\int_I \int_{\R^d} \rh(A_0(t)\ga A_1(t)^*)(x) f(t,x) dx dt}
		&\leq \abs{\int_I \Tr (A_0(t) \ga A_1(t)^* f(t,x)) dt }  \\
		&\leq \abs{\int_I \Tr (A_1(t)^* f(t,x)A_0(t) \ga ) dt} \\
		&\leq \No{A_1^\os f A_0^\oc }_{\FS^{\al'}} \|\ga\|_{\FS^\al} \\
		&\leq  \|\ga\|_{\Sa} \|A_1^\os f_1\|_{\FS^{2\al_1'}_{(t,x) \to x}} \|f_0 A_0^\oc\|_{\FS^{2\al_0'}_{x\to(t,x)}}\\
		&\leq \sqrt{C_0C_1}\|\ga\|_\Sa \|f\|_{L^{p'}_tL^{q'}_x}.
	\end{align}
\end{proof}

\subsection{Resolvent expansion of the propagators}
If $\si(\mu,\nu):= \frac{2}{\mu}+\frac{d}{\nu}=2$, $\mu < \I$ and $V \in L^\mu_t(\R, L^\nu_x)$, there exists a unitary propagator $U_V(t)$.
Namely, $U_V$ satisfies
\begin{align} \label{Duh1}
	U_V(t) = U(t) - i \int_0^t U(t-\ta)V(\ta)U_V(\ta)d\ta.
\end{align}
Let $V_0 := |V|^{1/2}$ and $V = V_0V_1$.
Multiplication with $V_0$ yields
\begin{align}\label{Duh2}
	V_0(t) U_V(t) &= V_0(t) U(t) - i \int_0^t V_0(t)U(t-\ta)V_1(\ta)V_0(\ta)U_V(\ta)d\ta.
\end{align}
Define $\CD(V_0, V_1)$ by 
\begin{align} \label{def of D}
	(\CD(V_0, V_1)f)(t,x) := -i\int_0^t V_0(t)U(t-\ta)(V_1(\ta)f(\ta))d\ta.
\end{align}
Then we can write \eqref{Duh1} and \eqref{Duh2} as
\begin{align} 
	U_V^\ocircle &= U^\ocircle + \CD(1, V_1)V_0 U_V^\ocircle,  \label{reDuh1}\\
	V_0 U_V^\ocircle &= V_0 U^\ocircle + \CD(V_0, V_1)V_0 U_V^\ocircle \nonumber. 
\end{align}
Hence, we have
\begin{align}
	(1 - \CD(V_0, V_1))V_0 U_V^\oc = V_0 U^\oc.
\end{align}
If $\CR(V_0, V_1) := (1-\CD(V_0, V_1))^{-1} \in B(L^2_{t,x})$, we have
\begin{align} \label{V0UV}
	V_0 U_V^\oc = \CR(V_0, V_1) V_0 U^\oc .
\end{align}
Therefore, \eqref{reDuh1} and \eqref{V0UV} imply
\begin{align} \label{UVexp}
	U_V^\oc = U^\oc + \CD(1, V_1) \CR(V_0, V_1) V_0 U^\oc.
\end{align}

The following lemma justifies \eqref{UVexp}.
\begin{lemma} \label{Resol est}
	Let $d \geq 1$.
	Let $\mu \in [1,\I)$ and $\nu \in [1,\I]$ satisfy $\si(\mu,\nu)=2$.
	Then there exists a monotone increasing function $\ph:[0,\I) \to [0,\I)$
	such that for any $V \in L^\mu_t (\R,L^\nu_x)$
	\begin{align}
		\|\CR(V_0, V_1)\|_{\CB(L^2_\tx)}
		\leq \ph(\|V\|_{L^\mu_t L^\nu_x})
	\end{align}
	holds, where $V_0 := |V|^{1/2}$ and $V = V_0 V_1$.
\end{lemma}

\begin{remark}
	In this paper, we denote monotone increasing functions by $\ph$.
	We use the same symbol $\ph$ for different monotone increasing functions. 
	This is like the constant $C$.
	We do not distinguish a monotone increasing function from another one.
\end{remark}

\begin{proof}
	It follows from the standard Strichartz estimates that 
	\begin{align}
		\| \CD(V_0,V_1) \|_{\CB(L^2_{t,x})} \leq C_0 \|V\|_{L^\mu_tL^\nu_x}.
	\end{align}
	Hence, if $\|V\|_{L^p_tL^q_x} \leq 1/(2C_0) =: \de$,
	then we have
	\begin{align}
		&\| \CD(V_0,V_1) \|_{\CB(L^2_{t,x})} \leq \frac{1}{2}, \\
		&\|\CR(V_0,V_1)\|_{\CB(L^2_{t,x})} = \|(1-\CD(V_0,V_1))^{-1}\|_{\CB(L^2_{t,x})} \leq 2.
	\end{align}
	Thus, we can assume that $\|V\|_{L^\mu_t L^\nu_x} \geq \de$.
	There exists $0=T_0 <T_1 < \cdots < T_N<T_{N+1}=\I$
	such that
	\begin{align}\label{V condi}
		&\frac{\de}{2} < \|V\|_{L^\mu_t(I_k, L^\nu_x)} < \de \mbox{ for } k = 0,\dots,N-1, \\
	    &\|V\|_{L^\mu_t(I_N, L^\nu_x)} < \de,
	\end{align}
	where $I_k = (T_k, T_{k+1})$.
	Therefore, we can solve the equation
	\begin{align} \label{u-f Eq}
		(1-\CD(V_0,V_1))u = f 
	\end{align}
	uniquely on $I_0$.
	We denote this solution by $u_0 \in L^2_t(I_0,L^2_x)$.
	Assume that there exists a unique solution on $[0,T_k] = I_0\cup \dots \cup I_{k-1}$,
	denoted by $u_{k-1} \in L^2_t([0,T_k],L^2_x)$.
	We consider the equation of $u$ on $I_k = [T_k,T_{k+1}]$:
	\begin{align} \label{u Eq}
		u(t) +i\int_{T_k}^t V_0(t)U(t-\ta)V_1(\ta)u(\ta)d\ta
		= f(t) - i\int_0^{T_k} V_0(t)U(t-\ta)V_1(\ta)u_{k-1}(\ta) d\ta.
	\end{align}
	This is solvable; that is, a unique solution exists $u_k \in L^2_t(I_k,L^2_x)$ to \eqref{u Eq}.
	Let $u_k(t)=u_{k-1}(t)$ for $t \in [0,T_k]$.
	Then $u_k \in L^2_t([0,T_{k+1}], L^2_x)$ is the solution to \eqref{u-f Eq}.
	We have a unique solution to \eqref{u-f Eq} on $u \in L^2_t([0,\I),L^2_x)$ inductively.
	We can similarly extend this solution to $u\in L^2_t(\R, L^2_x)$.
	
	Now we estimate the operator norm of $\CR(V_0, V_1)$.
	We have
	\begin{align}
		\|\CR(V_0,V_1)f\|_{L^2_{t,x}} &= \|u\|_{L^2_t(\R,L^2_x)} \\
		&\leq \|u\|_{L^2_t([0,\I),L^2_x)} + \|u\|_{L^2_t((-\I,0],L^2_x)}.
	\end{align}
	By \eqref{u Eq}, we have
	\begin{align} \label{u_k ineq}
		\|u_k\|_{L^2_t(I_k,L^2_x)}
		\ls \|f\|_{L^2_{t,x}} + C_0\|V\|_{L^\mu_t L^\nu_x} \|u_{k-1}\|_{L^2_t([0,T_k],L^2_x)};
	\end{align}
	hence we obtain
	\begin{align}
		\|u_k\|_{L^2_t([0,T_{k+1}],L^2_x)}
		\ls \|f\|_{L^2_{t,x}} + \Br{C_0\|V\|_{L^\mu_t L^\nu_x}}\|u_{k-1}\|_{L^2_t([0,T_k],L^2_x)}.
	\end{align}
	Hence it follows from \eqref{u_k ineq} that
	\begin{align}
		\|u\|_{L^2_t([0,\I),L^2_x)} \ls \bra{C_0 \|V\|_{L^\mu_tL^\nu_x}}^N \|f\|_{L^2_{t,x}}.
	\end{align}
	Since
	\begin{align}
		N \K{\frac{\de}{2}}^\mu \leq \|V\|_{L^\mu_t(\R,L^\nu_x)}^\mu,
	\end{align}
	we have
	\begin{align}
		\|u\|_{L^2_t([0,\I),L^2_x)} 
		\lesssim \bra{C_0\|V\|_{L^\mu_tL^\nu_x}}^{(2/\de)^{\mu}\|V\|_{L^\mu_tL^\nu_x}^\mu} \|f\|_{L^2_{t,x}}.
	\end{align}
	We can bound $\|u\|_{L^2((-\I,0],L^2_x)}$ similarly.
	From the above, we have
	\begin{align}
		\|\CR(V_0,V_1)\|_{\CB(L^2_{t,x})}
		\ls \bra{C_0\|V\|_{L^\mu_tL^\nu_x}}^{C\|V\|_{L^\mu_tL^\nu_x}^\mu} .
	\end{align}
\end{proof}

\subsection{Propagators with time-dependent potentials}
Let $p,q \in (1,\I)$ and $\al \in [1,\I]$.
We call $(p,q,\al)$ admissible when $\frac{1}{\al} \ge \frac{1}{dp} + \frac{1}{q}$ and $\al < p$ hold.
The following result is the most essential tool in this paper:
\begin{th.}[\cite{2014FLLS}Theorem 1; \cite{2017FS}Theorem 8; \cite{2019BHLNS}Theorem 1.5] \label{Ostri}
	Let $d \ge 1$.
	Let $(p,q,\al)$ be admissible and $\si(p,q)=\frac{2}{p}+\frac{d}{q}=d-2s$ for $s \in (0,d/2)$.
	Then it holds that
	\begin{align}
		\|\rh(U(t)_\st Q_0)\|_{L^p_t L^q_x} \ls \|\sd^s_\st Q_0\|_{\FS^\al}.
	\end{align}
\end{th.}

By Theorems \ref{SCK}, \ref{Ostri} and Lemma \ref{DP}, we obtain the following corollary.
\begin{corollary} \label{DS}
	Let $d \geq 1$.
	Let $(p_n,q_n,\al_n)$ be admissible and $\si(p_n,q_n)=d$ for $n = 0,1$.
	Define
	\begin{align}
		\CD(g_0,g_1) = \int_0^t g_0(t)U(t-\ta)g_1(\ta)d\ta.
	\end{align}
	Then we have
	\begin{align}
		\|\CD(g_0,g_1)\|_{\FS^{\al'}_{t,x}} \ls \prod_{n=0,1} \|g_n\|_{L^{2p_n'}_t L^{2q_n'}_x},
	\end{align}
	where $\frac{1}{\al'} = \frac{1}{2\al_0'} + \frac{1}{2\al_1'}$.
\end{corollary}

\begin{proof}
	First, we consider the operator
	\begin{align}
		\BT(g_0,g_1) = \int_0^\I g_0(t)U(t-\ta)g_1(\ta)d\ta = g_0 U^\oc U^\os g_1.
	\end{align}
	By Lemma \ref{DP} and Theorem \ref{Ostri}, we obtain
	\begin{align}
		\|\BT(g_0,g_1)\|_{\FS^{\al'}_{\tx}} &\le \|g_0 U^\oc\|_{\FS^{2\al_0'}_{x \to (\tx)}} \|U^\os g_1\|_{\FS^{2\al_1'}_{(\tx)\to x}}  \\
		&\ls \prod_{n=0,1} \|g_n\|_{L^{2p_n'}_t L^{2q_n'}_x}.
	\end{align}
	Therefore, we get the desired estimate by Theorem \ref{SCK}.
\end{proof}

We extend Theorem \ref{Ostri} with $s=0$ as follows.
\begin{lemma} \label{U_V-U}
	Let $d\geq 1$ and $0 \in I \subset \R$ be an interval.
	Let $(p,q,\al)$ be admissible and $\si(p,q)=d$.
	Let $\mu \in [1, \I)$ and $\nu \in [1, \I]$ satisfy $\si(\mu,\nu)=2$.
	Then there exists a monotone increasing function $\ph:[0,\I) \to [0,\I)$ such that 
	\begin{align} \label{U_V-U est}
		\|\rh(U_V(t)_\st \ga)\|_{L^p_t(I,L^q_x)} \le \ph(\|V\|_{L^\mu_t(I,L^\nu_x)}) \|\ga\|_{\FS^\al}.
	\end{align}
\end{lemma}

\begin{proof}
	We use Lemma \ref{DP}.
	Let $f \in L^{2p'}_t(I,L^{2q'}_x)$.
	By \eqref{UVexp}, we have
	\begin{align}
		\|f U_V^\oc\|_{\Sxtx} \le \|f U^\oc\|_{\Sxtx} + \|\CD(f,V_1) \CR(V_0,V_1) V_0 U^\oc\|_{\Sxtx}.
	\end{align}
	On the one hand, Lemma \ref{DP} and Theorem \ref{Ostri} yield
	\begin{align}
		\|f U^\oc\|_{\Sxtx} \ls \|f\|_{L^{2p'}_t(I,L^{2q'}_x)}.
	\end{align}
	On the other hand, we have by Lemma \ref{Resol est}
	\begin{align}
		\|\CD(f,V_1) \CR(V_0,V_1) V_0 U^\oc\|_{\Sxtx}
		&\le \|\CD(f,V_1)\|_{\FS^{2\al'}_\tx} \|\CR(V_0,V_1)\|_{\CB(L^2_x)} \|V_0 U^\oc\|_{\CB(L^2_x \to L^2_\tx)} \\
		&\le \|f\|_{L^{2p'}_t(I,L^{2q'}_x)} \|V\|_{L^\mu_t L^\nu_x} \ph(\|V\|_{L^\mu_t L^\nu_x}),
	\end{align}
	where we estimated $\|\CD(f,V_1)\|_{\FS^{2\al'}_\tx}$ in the same way as the proof of Corollary \ref{DS}.
\end{proof}

Note that the ``usual" Strichartz estimates also hold:
\begin{lemma} \label{UVStri}
	Let $d \geq 3$.
	Let $p,q,\wt{p},\wt{q} \in [2, \I]$ satisfy $\si(p,q) = \si(\wt{p},\wt{q})=d/2$.
	Let $\mu \in [1, \I)$ and $\nu \in [1, \I]$ satisfy $\si(\mu,\nu) = 2$.
	Then there exists a monotone increasing function $\ph:[0,\I) \to [0, \I)$ such that
	\begin{align}
		&\|U_V(t)u_0\|_{L^p_tL^q_x}
		\leq \ph(\|V\|_{L^\mu_t L^\nu_x}) \|u_0\|_{L^2_x}, \label{UVStri1}\\
		&\No{ \int_\R U_V(t,\ta)f(\ta) d\ta }_{L^p_tL^q_x}
		\leq \ph(\|V\|_{L^\mu_t L^\nu_x}) \|f\|_{L^{\wt{p}'}_t L^{\wt{q}'}_x}, \label{UVStri2} \\
		&\No{ \int_0^t U_V(t,\ta)f(\ta) d\ta}_{L^p_tL^q_x}
		\leq \ph(\|V\|_{L^\mu_t L^\nu_x}) \|f\|_{L^{\wt{p}'}_t L^{\wt{q}'}_x}. \label{UVStri3}
	\end{align}
\end{lemma}

\begin{proof}
	The usual Strichartz estimates and \eqref{UVexp} immediately imply \eqref{UVStri1}.
	\eqref{UVStri1} and usual duality argument imply \eqref{UVStri2}.
	By the same argument that led to \eqref{UVexp}, we have
	\begin{align} \label{UVts}
		U_V(t,\ta)f = U(t-\ta)f - i \int_\ta^t U(t-\ta_1) V_1(\ta_1) (Tf)(\ta_1,\ta) d\ta_1, 
	\end{align}
	where $(Tf)(\ta_1,\ta)= \CR(V_0,V_1)(V_0U(\cdot-\ta)f)(t)$.
	The standard Strichartz estimates and \eqref{UVts} imply \eqref{UVStri3}.
\end{proof}

The following lemma is also useful:
\begin{lem.} \label{almost commu}
	Let $d \ge 3$.
	Let $s \in [0,\I)$.
	Let $\mu, \nu,\wt{\nu} \in (1, \I)$ satisfy $\frac{1}{\nu} = \frac{1}{\wt{\nu}} + \frac{s}{d}$ and $\si(\mu,\wt{\nu})=2$.
	Let $(p,q)$ be an admissible pair of the standard Strichartz estimate; that is, $p, q \in [2,\I]$ and $\si(p,q) = d/2$.
	Then there exists a monotone increasing function $\ph:[0,\I) \to [0,\I)$ such that
	\begin{align}
		\|\sd^s U_V(t) \sd^{-s} u_0\|_{L^p_t L^q_x} \le \ph(\|V\|_{L^\mu_t H^{s,\nu}_x}) \|u_0\|_{L^2_x}.
	\end{align}
\end{lem.}

\begin{proof}
	First, we assume
	\begin{align}
		\tw \le \frac{1}{p} + \frac{1}{\mu} \le 1. 
	\end{align}
	Since $V \in L^\mu_t H^{s,\nu}_x \hookrightarrow L^\mu_t L^{\wt{\nu}}_x$ and $\si(\mu,\wt{\nu})=2$, 
	the propagator $U_V(t)$ is well-defined.
	We have
	\begin{align}
		\|\sd^s U_V(t)\sd^{-s}u_0\|_{L^p_t L^q_x}
		&\le \|U(t)u_0\|_{L^p_t L^q_x} + \No{\int_0^t U(t-\ta)\sd^s (V(\ta)U_V(\ta) \sd^{-s} u_0) d\ta}_{L^p_t L^q_x} \\
		&\le C_0\|u_0\|_{L^2_x} + C_0\No{\sd^s \K{ V(t)U_V(t) \sd^{-s} u_0}}_{L^{r'}_t L^{c'}_x} \\
		&\le C_0\|u_0\|_{L^2_x} + CC_0\|V\|_{L^\mu_t H^{s,\nu}_x} \|\sd^s U_V(t) \sd^{-s} u_0\|_{L^p_t L^q_x},
	\end{align}
	where
	\begin{align}
		\frac{1}{r}:= 1-\frac{1}{\mu}-\frac{1}{p},
		\quad \frac{1}{c} := 1 - \frac{1}{\wt{\nu}} - \frac{1}{q}.
	\end{align}
	$(r,c)$ is an admissible pair of the standard Strichartz estimates.
	Therefore, we get
	\begin{align}
		\|\sd^s U_V(t)\sd^{-s}u_0\|_{L^p_t L^q_x} \le 2C_0 \|u_0\|_{L^2_x}
	\end{align}
	when $\|V\|_{L^\mu_t H^{s,\nu}_x} \le \de$ with sufficiently small $\de>0$.
	Let $\|V\|_{L^\mu_t([0,\I),H^{s,\nu}_x)} > \de$.
	Then there exist $0=T_0 < T_1 < \cdots < T_N < T_{N+1} = \I$ such that
	\begin{align}
		&\frac{\de}{2} \le \|V\|_{L^\mu_t(I_n, H^{s,\nu}_x)} \le \de, \mbox{ for } 0 \le n \le N; \label{N}\\
		&\|V\|_{L^\mu_t(I_{N+1}, H^{s,\nu}_x)} \le \de, \nonumber
	\end{align}
	where $I_n := [T_n,T_{n+1}]$.
	First, we have
	\begin{align}
		\|\sd^sU_V(t)\sd^{-s} u_0\|_{L^p_t(I_0,L^q_x)}
		\le 2C_0
		    \|u_0\|_{L^2_x}.
	\end{align}
	Let $V_1(t):= V(t) \ch_{I_1}(t)$. 
	Since $U_V(t,T_1) = U_{V_1}(t,T_1)$ for $t \in I_1$ and $U_{V_1}(t) = U(t)$ for $t \in I_0$, we have
	\begin{align}\label{UV ext}
		U_V(t) = U_{V_1}(t,T_1)U_V(T_1) = U_{V_1}(t) U(-T_1) U_V(T_1). 
	\end{align}
	Hence, we have
	\begin{align} 
		\|\sd^s U_V(t) \sd^{-s} u_0\|_{L^p_t(I_1,L^q_x)}
		&\le 2C_0\|\sd^s U_V(T_1)\sd^{-s} u_0\|_{L^2_x} \\
		&\le (2C_0)^2 \|u_0\|_{L^2_x}.
	\end{align}
	In the same way, we obtain
	\begin{align}
		\|\sd^s U_V(t) \sd^{-s} u_0\|_{L^p_t(I_n,L^q_x)}
		\le (2C_0)^{n+1} \|u_0\|_{L^2_x},
	\end{align}
	which implies
	\begin{align}
		\|\sd^s U_V(t) \sd^{-s} u_0\|_{L^p_t([0,\I),L^q_x)}
		&\le \sum_{n=0}^N (2C_0)^{n+1} \|u_0\|_{L^2_x}  \\
		&\le N(2C_0)^{N+1} \|u_0\|_{L^2_x}.
	\end{align}
	$N \ls \|V\|_{L^\mu_t H^{s,\nu}_x}^\mu$ yields the desired estimates.
	
	Next, we consider the case $\frac{1}{\mu} + \frac{1}{p} \le \tw$.
Let $\frac{1}{r}=\frac{1}{2}-\frac{1}{d}$.
Then the endpoint Strichartz estimates (\cite{1998KT}) imply
\begin{align}
	\|\sd^s U_V(t)\sd^{-s} u_0\|_{L^p_t L^q_x}
	&\le \|U(t)u_0\|_{L^p_t L^q_x}
	+ \No{ \int_0^t U(t-\ta)\sd^s (V(\ta)U_V(\ta) \sd^{-s}u_0)d\ta }_{L^p_t L^q_x} \\
	&\ls \|u_0\|_{L^2_x} + \|\sd^s (V(t)U_V(t)\sd^{-s}u_0)\|_{L^2_t L^{r'}_x} \\
	&\ls \|u_0\|_{L^2_x} + \|\sd^s V\|_{L^\mu_t L^\nu_x} \|\sd^s U_V(t)\sd^{-s} u_0 \|_{L^{\wt{p}}_t L^{\wt{q}}_x},
\end{align}	
where $(\wt{p}, \wt{q})$ is an admissible pair of the standard Strichartz estimates.
Since $\frac{1}{\mu}+\frac{1}{\wt{p}} = \tw$, we conclude that
\begin{align}
	\|\sd^s U_V(t)\sd^{-s} u_0\|_{L^p_t L^q_x}
	\le \ph(\|V\|_{L^\mu_t L^\nu_x}) \|u_0\|_{L^2_x}.
\end{align}

Finally, we consider the case $\frac{1}{\mu}+\frac{1}{p} \ge 1$.
We have
\begin{align}
	\|\sd^s U_V(t)\sd^{-s} u_0\|_{L^p_t L^q_x}
	&\le \|U(t)u_0\|_{L^p_t L^q_x}
	+ \No{ \int_0^t U(t-\ta)\sd^s (V(\ta)U_V(\ta) \sd^{-s}u_0)d\ta }_{L^p_t L^q_x} \\
	&\ls \|u_0\|_{L^2_x} + \|\sd^s (V(t)U_V(t)\sd^{-s}u_0)\|_{L^1_t L^2_x} \\
	&\ls \|u_0\|_{L^2_x} + \|\sd^s V\|_{L^\mu_t L^\nu_x} \|\sd^s U_V(t)\sd^{-s} u_0 \|_{L^{\wt{p}}_t L^{\wt{q}}_x},
\end{align}	
where $(\wt{p}, \wt{q})$ is an admissible pair of the standard Strichartz estimates.
Since $\frac{1}{\mu}+\frac{1}{\wt{p}} = 1$, we conclude that
\begin{align}
	\|\sd^s U_V(t)\sd^{-s} u_0\|_{L^p_t L^q_x}
	\le \ph(\|V\|_{L^\mu_t L^\nu_x}) \|u_0\|_{L^2_x}.
\end{align}

\end{proof}

\section{Strichartz estimates for density functions}
In this section, we prove the key estimates in this paper.
\begin{th.}\label{Key}
	Let $d \ge 3$.
	Let $0 \le \wt{s} \le s \le 1$.	
	Let $p,q,\mu,\nu,\al \in (1,\I)$ satisfy $\si(p,q)=d-s$ and $\si(\mu,\nu) = 2+\wt{s}$.
	Let $(p,q,\al)$ be admissible.
	Assume that there exist $p_j,q_j,\al_j \in (1,\I)$ for $j=0,1,2$
	and $\mu_j,\nu_j \in (1,\I)$ for $j=0,1$ such that
	\begin{align}
		&\si(p_j,q_j) = d-j,\quad (p_j,q_j,\al_j) \mbox{ are admissible for } j=0,1,2, \\
		&\si(\mu_0,\nu_0) = 2, \quad \si (\mu_1,\nu_1) = 2+s_1 := 2 + \frac{\wt{s}}{s}, \quad 1+s_1 < \frac{2}{\mu_1}+\frac{d}{2},\\
		&\frac{1}{X} = \frac{1-s}{X_0} + \frac{s}{X_1}, \quad X = p,q,\mu,\nu,\al,
		\quad \frac{1}{Y_1} = \frac{1}{2} \K{\frac{1}{Y_0} + \frac{1}{Y_2}}, \quad Y=p,q,\al.
	\end{align}
Then there exists a monotone increasing function $\ph:[0,\I) \to [0,\I)$ such that
	\begin{align}
		&\|\rh(U_V(t)_\st Q_0)\|_{L^p_t H^{s,q}_x}
		\le \ph(\|V\|_{L^{\mu}_t H^{s,\nu}_x})
		\|Q_0\|_{\CH^{s,\al}}, \label{g stri}\\
		&\|\rh(U_V(t)_\st Q_0) - \rh(U_W(t)_\st Q_0)\|_{L^p_t H^{s,q}_x} \\
		&\quad \le \ph(\|V\|_{L^{\mu}_t H^{s,\nu}_x} + \|W\|_{L^{\mu}_t H^{s,\nu}_x})
		\|V-W\|_{L^\mu_t H^{s,\nu}_x}
		\|Q_0\|_{\CH^{s,\al}}. \label{diff g stri}
	\end{align}
\end{th.}

The following corollary is enough for our analysis:
\begin{cor.} \label{key cor}
	Let $d=3$. 
	Then there exists a monotone increasing function $\ph:[0,\I) \to [0,\I)$ such that
	\begin{align}
		&\|\rh(U_V(t)_\st Q_0) \|_{L^2_t H^{1/2}_x} \le \ph(\|V\|_{L^2_t H^{1/2}_x}) \|Q_0\|_{\CH^{1/2,3/2}},\\
		&\|\rh(U_V(t)_\st Q_0) - \rh(U_W(t)_\st Q_0)\|_{L^2_t H^{1/2}_x} \\
		&\quad \le \ph(\|V\|_{L^2_t H^{1/2}_x} + \|W\|_{L^2_t H^{1/2}_x})
		\|V-W\|_{L^2_t H^{1/2}_x}
		\|Q_0\|_{\CH^{1/2,3/2}}.
	\end{align}
\end{cor.}

\begin{proof}
	Define exponents as follows:
	\begin{align}
		&s=\wt{s}=\tw, \quad p=q=\mu=\nu=2, \quad \al=\frac{3}{2}, \\
		&p_0 = q_0 = \frac{5}{3}, \quad \al_0 = \frac{5}{4},
		\quad p_1=q_1=\frac{5}{2}, \quad \al_1 = \frac{15}{8},
		\quad p_2=q_2=5, \quad \al_2=\frac{15}{4}, \\
		&\mu_0 = \mu_1=2, \quad \nu_0 = 3, \quad \nu_1 = \frac{3}{2}.
	\end{align}
	Then Theorem \ref{Key} implies the result.
\end{proof}

Before proving the above theorem, we define some notations.
For $\vF=(F_1,\dots, F_n) \in \CS(\BR^d)^n$, define
\begin{align}
	\CW_n [\vF](t)
	= (-i)^n \int_0^t dt_1 \int_0^{t_1} dt_2 \cdots \int_0^{t_{n-1}} dt_n
	\CU[F_1](t_1) \cdots \CU[F_n](t_n),
\end{align}
where $\CU[V](t):= U(t)^*V(t)U(t)$.
And we define multilinear operators by
\begin{align}
	T_{n,m}[\vF,Q_0, \vG](t) :=  \rh \K{ U(t)\CW_n[F](t) Q_0 \CW_m [G](t)^* U(t)^* }
\end{align}
for $\vF \in \CS(\BR^d)^n$ and $\vG \in \CS(\BR^d)^m$.
We have

\begin{lem.}\label{muliti est} Under the same assumptions as in Theorem \ref{Key}, 
	The following multilinear estimates hold:
	\begin{align}
	&\|T_{n,m} [\vF,Q_0,\vG]\|_{L^p_t H^{s,q}_x}
	\ls C_0^{n+m}
	\prod_{j=1}^n \|F_j\|_{L^{\mu}_t H^{s,\nu}_x}
	\|Q_0\|_{\CH^{s,\al}_x}
	\prod_{k=1}^m \|G_k\|_{L^{\mu}_t H^{s,\nu}_x}. 
\end{align}	
\end{lem.}

\begin{proof}
	It suffices to prove the following two estimates:
		\begin{align}
		&\|T_{n,m} [\vF,Q_0,\vG]\|_{L^{p_0}_t L^{q_0}_x}
		\ls C_0^{n+m}
		\prod_{j=1}^n \|F_j\|_{L^{\mu_0}_t L^{\nu_0}_x}
		\|Q_0\|_{\FS^{\al_0}_x}
		\prod_{k=1}^m \|G_k\|_{L^{\mu_0}_t L^{\nu_0}_x} \label{L2},
		\\
		&\|T_{n,m} [\vF,Q_0,\vG]\|_{L^{p_1}_t H^{1,q_1}_x}
		\ls C_0^{n+m}
		\prod_{j=1}^n \|F_j\|_{L^{\mu_1}_t H^{1,\nu_1}_x}
		\|Q_0\|_{\CH^{1, \al_1}_x}
		\prod_{k=1}^m \|G_k\|_{L^{\mu_1}_t H^{1,\nu_1}_x}. \label{H1}
	\end{align}	
	We only prove \eqref{H1} because \eqref{L2} can be shown similarly.
	We have 
	\begin{align}
		\|T_{n,m}[\vF,Q_0,\vG]\|_{L^{p_1}_t H^{1,q_1}_x}
		&\sim \|T_{n,m}[\vF,Q_0,\vG]\|_{L^{p_1}_t L^{q_1}_x} + \|\na T_{n,m}[\vF,Q_0,\vG]\|_{L^{p_1}_t L^{q_1}_x}.
	\end{align}
	Let $f \in L^{2p_1'}_t L^{2q_1'}_x$.
	Note that $\sd = \sdi - L \na$ with $L:= \na \sdi$.
	We have
	\begin{align}
		f(t)U(t)\CW_n[\vF](t) \sd^{-1} 
		&= f(t)U(t)\sdi (\sdi - L\na)\CW_n[\vF](t) \sd^{-1} \\
		&=:W_1 - W_2.
	\end{align}
	We can assume $n \ge 1$ because when $n=0$, we obtain \eqref{half} by Theorem \ref{Ostri}.
	$W_1$ is a good term, and the problem is $W_2$.
	We have
	\begin{align}
		W_2 = f(t)U(t)\sdi L \sum_{k=1}^n \CW_n[\vF^k](t) \sd^{-1} 
		+ f(t)U(t)\sdi L \CW_n[\vF](t) \na \sd^{-1},
	\end{align}
	where $\vF^k = (F_1, \cdots, F_{k-1}, \na F_k, F_{k+1}, \cdots, F_n)$.
	The last term is harmless.
	Let $F_{1,0}:= |F_1|^{1/2}$ and $F_1 = F_{1,0}F_{1,1}$.
	Note that $L^{\mu_1}_t H^{1,\nu_1}_x \hookrightarrow L^{\mu_1}_t L^{\wt{\nu_1}}_x$ with $\si(\mu_1, \wt{\nu_1}) =2$.
	Then we have
	\begin{align}
		&\No{ f(t)U(t)\sdi L
			\int_0^t dt_1 
			\int_0^{t_1} dt_2 \cdots \int_0^{t_{n-1}} dt_n
			\CU[F_1](t) \cdots \CU[\na F_k](t_k) \cdots \CU[F_n](t_n) \sdi}_{\FS^{2\al_1'}_{x \to (\tx)}} \\
		&\quad \le \No{f(t)U(t)\sdi L \int_0^t dt_1 U(t_1)F_{1,0}(t_1)}_{\FS^{2\al_1'}_\tx} \\
		&\qquad \times \No{ F_{1,1}(t_1)U(t_1)\int_0^{t_1} dt_2 \cdots \int_0^{t_{n-1}} dt_n
			\CU[F_1](t) \cdots \CU[\na F_k](t_k) \cdots \CU[F_n](t_n) \sdi }_{\CB(L^2_x \to L^2_\tx)} \\
		&\quad \ls C_0
		\|f\|_{L^{2p_1'}_t L^{2q_1'}_x}
		\|F\|_{L^{\mu_1}_t L^{\wt{\nu_1}}_x}^{1/2} \bigg\|
		F_{1,1}(t_1)U(t_1)
		\int_0^{t_1} dt_2 \CU[F_2](t_1)\\
		&\qquad \times
			\cdots 
			\int_0^{t_{k-1}} dt_k \CU[\na F_k](t_k) 
			\cdots
			\int_0^{t_{n-1}} dt_n \CU[F_n](t_n) \sdi\bigg\|_{\CB(L^2_x \to L^2_\tx)} =: I.
	\end{align}
	To justify the last inequality, we used Theorem \ref{SCK}.
	Moreover, we have
	\begin{align}
		I& \le C_0^k 
		\|f\|_{L^{2p_1'}_t L^{2q_1'}_x}
		\prod_{j=1}^{k-1}\|F_j\|_{L^{\mu_1}_t L^{\wt{\nu_1}}_x}\\
		&\quad \times
		\No{
			(\na F_k)(t_k) U(t_k)
			\int_0^{t_k}dt_{k+1} \CU[F_{k+1}](t_{k+1})
			\cdots
			\int_0^{t_{n-1}} dt_n \CU[F_n](t_n) \sdi
		}_{\CB(L^2_x \to L^{\wt{r}'}_t L^{\wt{c}'}_x)} \\
		&\le
		C_0^k \|f\|_{L^{2p_1'}_t L^{2q_1'}_x}
		\prod_{j=1}^k \|F_j\|_{L^{\mu_1}_t H^{1,\nu_1}_x}\\
		&\quad \times
		\No{
			U(t_k)
			\int_0^{t_k}dt_{k+1} \CU[F_{k+1}](t_{k+1})
			\cdots
			\int_0^{t_{n-1}} dt_n \CU[F_n](t_n) \sdi
		}_{\CB(L^2_x \to L^r_t L^{c^*}_x)} \\
		&\le
		C_0^k \|f\|_{L^{2p_1'}_t L^{2q_1'}_x}
		\prod_{j=1}^k \|F_j\|_{L^{\mu_1}_t H^{1,\nu_1}_x}\\
		&\quad \times
		\No{
			\sd
			U(t_k)
			\int_0^{t_k}dt_{k+1} \CU[F_{k+1}](t_{k+1})
			\cdots
			\int_0^{t_{n-1}} dt_n \CU[F_n](t_n) \sdi
		}_{\CB(L^2_x \to L^r_t L^c_x)},
	\end{align}
	where $\wt{r},\wt{c},r,c$ are defined as follows:
	Fix $r \ge 2$ such that $\tw-\frac{1}{\mu_1} < \frac{1}{r} < \min(\frac{d}{4}-\frac{s_1}{2}, 1-\frac{1}{\mu_1})$.
	Define $\wt{r} \in [2,\I]$ by $\frac{1}{\wt{r}'}=\frac{1}{\mu_1}+\frac{1}{r}$.
	Then there exists $c,\wt{c} \in [2,\I]$ such that $\si(r,c) = \si(\wt{r},\wt{c}) = \frac{d}{2}$.
	By direct calculation, we obtain
	\begin{align}
		\frac{1}{c^*}:=\frac{1}{\wt{c}'}-\frac{1}{\nu_1} = \frac{1}{c}-\frac{s_1}{d} > 0.
	\end{align}
	Since
	\begin{align}
		\No{
			\int_0^t U(t-\ta)(g(\ta)h(\ta)) d\ta
		}_{L^r_t H^{1,c}_x}
		\le C_0 \|g\|_{L^{\mu_1}_t H^{1,\nu_1}_x} \|h\|_{L^r_t H^{1,c}_x},
	\end{align}
	we have
	\begin{align}
		&\No{
			\sd
			U(t_k)
			\int_0^{t_k}dt_{k+1} \CU[F_{k+1}](t_{k+1})
			\cdots
			\int_0^{t_{n-1}} dt_n \CU[F_n](t_n) \sdi
		}_{\CB(L^2_x \to L^r_t L^c_x)} \\
		&\quad \le
		C_0^{n-k}
		\prod_{j=k+1}^n \|F_j\|_{L^{\mu_1}_t H^{1,\nu_1}_x}.
	\end{align}
	From the above, we obtain
	\begin{align}
		\|f(t) U(t) \sdi L \CW_n[\vF^k](t) \sd^{-1}\|_{\FS^{2\al_1'}_{x \to (\tx)}}
		\ls C_0^n \|f\|_{L^{2p_1'}_t L^{2q_1'}_x}\prod_{j=1}^n \|F_j\|_{L^{\mu_1}_t H^{1,\nu_1}_x},
	\end{align}
	which yields
	\begin{align} \label{half}
		\|f(t)U(t)\CW_n[\vF](t) \sd^{-1} \|_{\FS^{2\al_1'}_{x \to (\tx)}}
		\ls C_0^n \|f\|_{L^{2p_1'}_t L^{2q_1'}_x}\prod_{j=1}^n \|F_j\|_{L^{\mu_1}_t H^{1,\nu_1}_x}.
	\end{align}
	Therefore, we conclude that
	\begin{align}
		\|T_{n,m}[\vF,Q_0,\vG]\|_{L^{p_1}_t L^{q_1}_x}
		\ls C_0^{n+m} \K{\prod_{j=1}^n
			\|F_j\|_{L^{\mu_1}_t H^{1,\nu_1}_x}}
		\|Q_0\|_{\CH^{1,\al_1}}
		\K{\prod_{k=1}^m \|G_k\|_{L^{\mu_1}_t H^{1,\nu_1}_x}}.
	\end{align}
	Finally, we estimate $\|\na T_{n,m}[\vF,Q_0,\vG]\|_{L^{p_1}_t L^{q_1}_x}$.
	In the completely same way as the above, we get
	\begin{align}
		&\|T_{n,m} [\vF,Q_0,\vG]\|_{L^{p_2}_t L^{q_2}_x}
		\ls C_0^{n+m}
		    \K{\prod_{j=1}^n \|F_j\|_{L^{\mu_1}_t H^{1,\nu_1}_x}}
		    \|Q_0\|_{\CH^{1,\al_2}}
		    \K{\prod_{k=1}^ m\|G_k\|_{L^{\mu_1}_t H^{1,\nu_1}_x}}, \label{End H1} \\
		&\|\rh\Dk{\na U(t)\CW_n[\vF](t) \sdi Q_0 \sdi \CW_n[\vG](t)^* U(t)^* \na} \|_{L^{p_0}_t L^{q_0}_x} \\
		&\quad \ls C_0^{n+m} \K{\prod_{j=1}^n \|F_j\|_{L^{\mu_1}_t H^{1,\nu_1}_x}}
		\|Q_0\|_{\FS^{\al_0}}
		\K{\prod_{k=1}^m \|G_k\|_{L^{\mu_1}_t H^{1,\nu_1}_x}}. \label{deriv L2}
	\end{align}
	Therefore, we obtain by Lemma \ref{asym}
	\begin{align}
		&\|\na T_{n,m}[\vF,Q_0,\vG]\|_{L^{p_1}_t L^{q_1}_x} \\
		&\quad \le \No{\rh \Dk{\na U(t) \CW_n[\vF](t) Q_0 \CW_m[\vG](t)^* U(t)}}_{L^{p_1}_t L^{q_1}_x} \\
		&\qquad + \No{\rh \Dk{U(t) \CW_n[\vF](t) Q_0 \CW_m[\vG](t)^* U(t) \na}}_{L^{p_1}_t L^{q_1}_x} \\
		&\quad \ls C_0^{n+m} \K{\prod_{j=1}^n \|F_j\|_{L^{\mu_1}_t H^{1,\nu_1}_x}} \|Q_0\|_{\CH^{1,\al_1}} \K{\prod_{k=1}^m \|G_k\|_{L^{\mu_1}_t H^{1,\nu_1}_x}}.
	\end{align}
	
\end{proof}

\begin{proof}[{\bf Proof of Theorem \ref{Key}}]
	First, we prove \eqref{g stri}.
	We have the series decomposition of $U_V(t)$:
	\begin{align}\label{UVWV}
		U_V(t) = \sum_{n=0}^\I U(t)\CW_V^{(n)}(t),
	\end{align}
	where $\CW_V^{(n)}(t) := \CW_n[\vV](t)$ and $\vV:=(V,\dots, V)$.
	If $\|V\|_{L^\mu_t H^{s,\nu}_x} < 1/(2C_0)=:\de$, Lemma \ref{muliti est} implies
	\begin{align}
		\|\sd^s \rh[U_V(t)_\st Q_0]\|_{L^p_t L^q_x} 
		&\le \sum_{n,m=1}^\I \|\sd^s \rh[U(t)\CW_V^{(n)}(t) Q_0 \CW_V^{(m)}(t)^* U(t)^*]\|_{L^p_t L^q_x} \\
&\ls \sum_{n,m=0}^\I
		C_0^{n+m}
		\|V\|_{L^\mu_t H^{s,\nu}_x}^{n+m}
		\|Q_0\|_{\CH^{s,\al}_x} \\
		&\le C
		\|Q_0\|_{\CH^{s,\al}_x}.
	\end{align}
	Let $\|V\|_{L^\mu_t([0,\I), H^{s,\nu}_x)} >\de$.
	We can divide the interval $[0,\I)$ such that
	\begin{align}
		&0=T_0 < T_1 < \cdots < T_N < T_{N+1}=\I, \quad I_k := [T_k,T_{k+1}] \mbox{ for } k=0, \dots N,\\
		&\frac{\de}{2} \le \|V\|_{L^\mu_t(I_k,H^{s,\nu}_x)} \le \de \mbox{ for } k=0,\dots,N-1, 
		\quad \|V\|_{L^\mu_t(I_N,H^{s,\nu}_x)} \le \de.
	\end{align}
	Define $V_k(t) := V(t) \ch_{I_k}(t)$.
	Then \eqref{UV ext}, Lemma \ref{almost commu} and $N \lesssim \|V\|_{L^\mu_t H^{s,\nu}_x}^\mu$ yield
	\begin{align}
		&\|\sd^s \rh[U_V(t)_\st Q_0]\|_{L^p_t([0,\I),L^q_x)} \\
		&\quad \le \sum_{k=0}^N
		\|\sd^s \rh[U_{V_k}(t)_\st (U_{V_k}(T_k)^*U_V(T_k)Q_0U_V(T_k)^*U_{V_k}(T_k))]\|_{L^p_t(I_k,L^q_x)} \\
		&\quad \le C \sum_{k=0}^N \|U(T_k)^*U_V(T_k)Q_0U_V(T_k)^*U(T_k)\|_{\CH^{s,\al}} \\
		&\quad \le C \sum_{k=0}^N \|\sd^s U_V(T_k) \sd^{-s}\|_{\CB}^2
		                    \|Q_0\|_{\CH^{s,\al}_x} \\
		&\quad \le \ph(\|V\|_{L^\mu_t H^{s,\nu}_x}) (N+1) \|Q_0\|_{\CH^{s,\al}_x} \\
		&\quad \le \ph(\|V\|_{L^\mu_t H^{s,\nu}_x}) \|Q_0\|_{\CH^{s,\al}_x}.
	\end{align}
	
	Finally, we prove \eqref{diff g stri}.
	It suffices to show that
	\begin{align}
		&\|\sd^s \rh((U_V(t)-U_W(t)) Q_0 U_V(t))\|_{L^p_t L^q_x} \nonumber \\
		&\quad \le \ph(\|V\|_{L^\mu_t H^{s,\nu}_x} + \|W\|_{L^\mu_t H^{s,\nu}_x})
		\|V-W\|_{L^\mu_t H^{s,\nu}_x}
		\|Q_0\|_{\CH^{s,\al}}.
	\end{align}
	We have
	\begin{align}
		&\|\sd^s \rh((U_V(t)-U_W(t)) Q_0 U_V(t))\|_{L^p_t L^q_x} \\
		&\quad \le \sum_{n,m=0}^\I
		\No{
		\sd^s \rh\K{U(t)(\CW_V^{(n)}(t) - \CW_W^{(n)}(t)) Q_0 \CW^{(m)}_V(t)^* U(t)^*}
		}_{L^p_t L^q_x} \\
		&\quad \le \sum_{n,m=0}^\I \sum_{k=1}^n
		\No{
			\sd^s \rh\K{U(t) \CW_n[\vX_k](t) Q_0 \CW_m[\vV](t)^* U(t)^*}
		}_{L^p_t L^q_x},
	\end{align}
	where $\vV= (V,\dots, V)$, $\vX_k=((\vX_k)_1, \dots, (\vX_k)_n)$ and 
	\begin{align}
		&(\vX_k)_j = V \mbox{ if } 1 \le j \le k-1, \\
		&(\vX_k)_j = V-W \mbox{ if } j = k, \\
		&(\vX_k)_j = W \mbox{ if } k+1 \le j \le n.
	\end{align}
By Lemma \ref{muliti est}, if $\|V\|_{L^\mu_t H^{s,\nu}_x}$ and $\|W\|_{L^\mu_t H^{s,\nu}_x}$ are sufficiently small, we obtain
\begin{align}
	&\|\sd^s \rh((U_V(t)-U_W(t)) Q_0 U_V(t))\|_{L^p_t L^q_x} \\
	&\quad \ls  \|V-W\|_{L^\mu_t H^{s,\nu}_x}
	           \|Q_0\|_{\CH^{s,\al}}
	           \sum_{n,m=0}^\I \sum_{k=1}^n
	           C_0^{m+n}
	           \|V\|_{L^\mu_t H^{s,\nu}_x}^{m+k-1}
	           \|W\|_{L^\mu_t H^{s,\nu}_x}^{n-k} \\
	&\quad \ls \|V-W\|_{L^\mu_t H^{s,\nu}_x} \|Q_0\|_{\CH^{s,\al}}.
\end{align}
The same argument as we used to prove \eqref{g stri} yields
\begin{align}
	\|\sd^s \rh((U_V(t)-U_W(t)) Q_0 U_V(t))\|_{L^p_t L^q_x}
	\le \ph(\|V\|_{L^\mu_t H^{s,\nu}_x})
	    \|V-W\|_{L^\mu_t H^{s,\nu}_x}
	    \|Q_0\|_{\FS^{s,\al}}.
\end{align}
\end{proof}

\section{Proof of Theorem \ref{den main}}
In this section, we give a proof of Theorem \ref{den main} by giving a nonlinear estimates:
\begin{lem.} \label{Lem}
	Let $d=3$. If $\|\lxr^2 f(\xi)\|_{L^1_\xi \cap L^\I_\xi} < \I$, 
	then there exists a monotone increasing function $\ph:[0,\I) \to [0,\I)$ such that 
	\begin{align}
	&\| \rh(N_1[V]) \|_{L^2_t H^{1/2}_x}
		\ls_f \ph(\|V\|_{L^2_t H^{1/2}_x})
		    \|V\|_{L^2_t H^{1/2}_x}^2 \label{BB}\\
	&\| \rh(N_2[V]) \|_{L^2_t H^{1/2}_x}
	\ls_f
	\ph(\|V\|_{L^2_t H^{1/2}_x})
	\|V\|_{L^2_t H^{1/2}_x}^2 \label{CC}\\
	&\| \rh(N_3[V]) \|_{L^2_t H^{1/2}_x}
	\ls_f 
	\ph(\|V\|_{L^2_t H^{1/2}_x})
	\|V\|_{L^2_t H^{1/2}_x}^3. \label{DD}
	\end{align}
	See \eqref{ABCD} for the definition of $N_1$, $N_2$ and $N_3$.
\end{lem.}

\begin{lem.} \label{LemB}
	Let $d=3$. If $\|\lxr^2 f(\xi)\|_{L^1_\xi \cap L^\I_\xi} < \I$, 
	then there exists a monotone increasing function $\ph:[0,\I) \to [0,\I)$ such that 
	\begin{align}
		&\| \rh(N_1[V]) - \rh(N_1[W]) \|_{L^2_t H^{1/2}_x} \\
		&\quad \ls_f
		\ph(\|V\|_{L^2_t H^{1/2}_x}+\|W\|_{L^2_t H^{1/2}_x})
		\K{\|V\|_{L^2_t H^{1/2}_x} + \|W\|_{L^2_t H^{1/2}_x}}
		\|V-W\|_{L^2_t H^{1/2}_x},
		\label{diff BB}\\
		&\| \rh(N_2[V]) - \rh(N_2[W])\|_{L^2_t H^{1/2}_x} \\
		&\quad \ls_f 
		\ph(\|V\|_{L^2_t H^{1/2}_x}+\|W\|_{L^2_t H^{1/2}_x})
		\K{\|V\|_{L^2_t H^{1/2}_x}+\|W\|_{L^2_t H^{1/2}_x}}
		\|V-W\|_{L^2_t H^{1/2}_x},
		\label{diff CC}\\
		&\| \rh(N_3[V]) - \rh(N_3[W]) \|_{L^2_t H^{1/2}_x} \\
		&\quad \ls_f 
		\ph(\|V\|_{L^2_t H^{1/2}_x}+\|W\|_{L^2_t H^{1/2}_x})
		\K{\|V\|_{L^2_t H^{1/2}_x}^2+\|W\|_{L^2_t H^{1/2}_x}^2}
		\|V-W\|_{L^2_t H^{1/2}_x}. \label{diff DD}
	\end{align}
\end{lem.}

First, we prove that Lemmas \ref{Lem} and \ref{LemB} imply Theorem \ref{den main}.
\begin{proof}
	Let
	\begin{align}
		&\Ph[\rh_Q](t):= (1+\CL_1)^{-1} (\rh(U_V(t)_\st Q_0) + \rh(N_1[V])+ \rh(N_2[V])+ \rh(N_3[V])), \\
		&E(R):= \{\rh_Q \in L^2_t(\R,H^\tw_x) : \|\rh_Q\|_{L^2_t H^{1/2}_x} \leq R\},
	\end{align}
	where $V = w \ast \rh_Q$ and $R>0$ will be taken later.
	We prove $\Ph:E(R) \to E(R)$ is a contraction map.
	Corollary \ref{key cor}, Lemma \ref{Lem} and Young's inequality imply
	\begin{align}
		\| \Ph[\rh_Q] \|_{L^2_t H^{1/2}_x}
		&\leq \|(1+\CL_1)^{-1}\|_{\CB(L^2_{t,x})}
		\left( \|\rh(U_V(t)_\st Q_0)\|_{L^2_t H^{1/2}_x}
			+ \|\rh(N_1[V])\|_{L^2_t H^{1/2}_x} \right.\\
		&\left.\qquad  + \|\rh(N_2[V])\|_{L^2_t H^{1/2}_x}
			+ \|\rh(N_3[V])\|_{L^2_t H^{1/2}_x}
			 \right)\\
		& \le C_f \ph(\|V\|_{L^2_t H^{1/2}_x})
		\K{\|V\|_{L^2_t H^{1/2}_x}^2 + \|V\|_{L^2_t H^{1/2}_x}^3 + \|Q_0\|_{\CH^{1/2,3/2}} } \\
		& \le C_f \ph(R)(R^2 + R^3 + \ep_0).
	\end{align}
	Hence, if $R >0$ and $\ep_0 >0$ are sufficiently small, then we obtain
	\begin{align}
		\| \Ph[\rh_Q] \|_{L^2_{t,x}} \leq R.
	\end{align}
	Therefore, the map $\Ph:E(R) \to E(R)$ is well-defined.
	We can prove the contraction and the uniqueness of the solution by Corollary \ref{key cor} and Lemma \ref{LemB}.
\end{proof}

In the rest of this section, we prove only \eqref{CC} because we can prove the other estimates in the same way.
Let 
\begin{align}
	N_2[V](t) &= i \int_0^t D_V(t,\ta) \ga_f V(\ta) U(\ta-t)d\ta-i\int_0^t D_V(t,\ta) V(\ta) \ga_f U(\ta-t)d\ta  \\
	&=:N_2^1[V](t) - N_2^2[V](t). \label{C1 C2}
\end{align}

\subsection{Estimate for $N_2^1$}
First, note that we have the wave operator decomposition
\begin{align}\label{eq:wave expansion}
	U_V(t,\ta) = U(t) \sum_{n=0}^\I \CW_V^{(n)}(t,\ta) U(\ta)^*,
\end{align}
where
\begin{align}
	&\CW_V^{(0)}(t,\ta) = I, \\
	&\CW_V^{(n)}(t,\ta) = (-i)^n \int_{\ta}^t dt_1 \cdots \int_{\ta}^{t_{n-1}} dt_n U(t) \CU[V](t_1) \cdots \CU[V](t_n) U(\ta)^*
\end{align}
and $\CU[F]:= U(t)^* F(t) U(t)$.
Hence, it follows that
\begin{align}
	\rh(N_2^1[V](t))
	&= \int_0^t \int_0^\ta \rh[U_V(t,\ta)V(\ta)U(\ta-\ta_1)  \ga_f V(\ta_1) U(\ta_1-t)] d\ta_1 d\ta  \\
	&= \sum_{n=0}^\I \int_0^t \int_0^\ta
	\rh \Big[U(t) \CW_V^{(n)}(t,\ta) U(\ta)^* V(\ta)U(\ta-\ta_1) \ga_f V(\ta_1) U(\ta_1-t) \Big] d\ta_1 d\ta \\
	&=: \sum_{n=0}^\I \CN_n(V).
\end{align}
Define
\begin{align}
	&\CW^{(n)}[g_1,\dots,g_n](t,\ta)
	:= (-i)^n \int_{\ta}^t dt_1 \int_{\ta}^{t_1} dt_2 \cdots \int_{\ta}^{t_{n-1}} dt_n
	U(\ta) \CU[g_1](t_1) \cdots \CU[g_n](t_n) U(\ta)^*,
\end{align}
and
\begin{align}
	&\CN_n[g_1, \dots, g_{n+2}](t) \\
	&\quad := \int_0^t d\ta \int_0^\ta d\ta_1
	\rh \Big[U(t) \CW^{(n)}[g_1,\cdots,g_n](t,\ta) U(\ta)^* g_{n+1}(\ta)U(\ta-\ta_1)
	\ga_f g_{n+2}(\ta_1) U(\ta_1-t) \Big] \\
	&\quad = (-i)^{n} \int_0^t dt_1 \int_0^{t_1}dt_2 \cdots \int_0^{t_{n-1}} dt_n \int_0^{t_n} d\ta \int_0^\ta d\ta_1 \\
	&\qquad \times \rh \Big[U(t) \CU[g_1](t_1) \cdots \CU[g_n](t_n) U(\ta)^*
	g_{n+1}(\ta) U(\ta-\ta_1)\ga_f g_{n+2}(\ta_1) U(\ta_1-t) \Big].
\end{align}

\subsubsection{With no derivative I}
Let $W \in C_c^\I((0,\I) \times \BR^3)$.
On the one hand, we have
\begin{align}
	&\abs{\int_0^\I \int_{\R^3}  W(t,x) \CN_n[g_1,\dots,g_{n+2}](t,x) dx dt } \\
	&\quad = \left| \Tr \left[ \int_0^\I  U(t)^* W(t,x) U(t) dt
	\int_0^t U(t_1)^* g_1(t_1) U(t_1)dt_1
	\cdots
	\int_0^{t_{n-1}} U(t_n)^* g_n(t_n) U(t_n)dt_n
	\right. \right. \\
	&\left. \left. \qquad \qquad \times
	\int_0^{t_n} U(\ta)^* g_{n+1}(\ta) U(\ta)d\ta
	\int_0^\ta U(\ta_1)^* \ga_f g_{n+2}(\ta_1) U(\ta_1)d\ta_1 \right] \right| \\
	&\quad = \left| \Tr \left[ \int_0^\I  U(t)^* W(t,x) U(t) dt \sdi
	\sd \int_0^t U(t_1)^* g_1(t_1) U(t_1)dt_1
	\times
	\cdots
	\right. \right. \\
	&\left. \left. \qquad \qquad \cdots \times
	\int_0^{t_n} U(\ta)^* g_{n+1}(\ta) U(\ta)d\ta
	\int_0^\ta U(\ta_1)^* \ga_f g_{n+2}(\ta_1) U(\ta_1)d\ta_1 \right] \right|.
\end{align}
Since $\sd = \sdi + \na \cdot \na \sdi$, we get
\begin{align}
	\abs{\int_0^\I \int_{\R^3}  W(t,x) \CN_n[g_1,\dots,g_{n+2}](t,x) dx dt }
	\le C_0^{n+2} \|W\|_{L^2_t L^2_x} \prod_{j=1}^{n+1}\|g_j\|_{L^2_t H^1_x} \|g_{n+2}\|_{L^2_t L^2_x}.
\end{align}
From the above, we obtain
\begin{align} \label{eq:inter. g}
	\|\CN_n[g_1,\dots,g_{n+2}]\|_{L^2_t L^2_x} \le C_0^{n+2} \prod_{j=1}^{n+1}\|g_j\|_{L^2_t H^1_x} \|g_{n+2}\|_{L^2_t L^2_x}.
\end{align}
On the other hand, we obtain
\begin{align}
	&\abs{\int_0^\I \int_{\R^3}  W(t,x) \CN_1[g_1,\dots,g_{n+2}](t,x) dx dt } \\
	&\quad = \left| \Tr \left[ \sdi \int_0^\I  U(t)^* W(t,x) U(t) dt \int_0^t U(t_1)^* g_1(t_1) U(t_1)dt_1
	\times
	\cdots
	\right. \right. \\
	&\left. \left. \qquad \qquad \cdots \times
	\int_0^{t_n} U(\ta)^* g_{n+1}(\ta) U(\ta)d\ta
	\int_0^\ta U(\ta_1)^* \ga_f g_{n+2}(\ta_1) U(\ta_1)d\ta_1 \sd \right] \right|.
\end{align}
Again, by the same argument as above, we obtain
\begin{align} \label{eq:inter. h}
	\|\CN_1[g_1,\dots,g_{n+2}]\|_{L^2_t L^2_x}
	\le C_0^{n+2} \prod_{j=1}^{n+1} \|g_j\|_{L^2_t L^2_x} \|g_{n+2}\|_{L^2_t H^1_x}.
\end{align}
By interpolating \eqref{eq:inter. g} and \eqref{eq:inter. h}, we obtain
\begin{align} \label{eq:no deriv}
	\|\CN_1[g_1,\dots,g_{n+2}]\|_{L^2_t L^2_x} \le C_0^{n+2} \prod_{j=1}^{n+2} \|g_j\|_{L^2_t H^{1/2}_x}.
\end{align}

\subsubsection{With no derivative II}
Next, we prove
\begin{align}\label{eq:no deriv II}
	\|\CN_1[g_1,\dots,g_{n+2}]\|_{L^{5/3}_\tx}
	\le C_0^{n+2} \prod_{j=1}^n
	\|g_j\|_{L^{2}_t L^3_x}
	\|g_{n+1}\|_{L^{20/11}_\tx} \|g_{n+2}\|_{L^{20/11}_\tx}.
\end{align}
Let $W \in C_c^\I((0,\I) \times \BR^3)$.
Let $W^0 := |W|^\tw$ and $W=W^0 W^1$.
Define $g_j^0 := |g|^{\frac{4}{11}}$ and $g_j =g_j^0 g_j^1$ for $j=n+1,n+2$.
Also, define $g_j^0:= |g_j|^\tw$ and $g_j = g_j^0 g_j^1$ for $j=1,\dots,n$.
Then, by Theorem \ref{SCK}, we have
\begin{align}
	&\abs{\int_0^\I \int_{\R^3}  W(t,x) \CN_1[g_1,\dots,g_{n+2}](t,x) dx dt } \\
	&\quad = \left| \Tr \left[ \int_0^\I  U(t)^* W(t,x) U(t) dt
	\int_0^t U(t_1)^* g_1(t_1) U(t_1)dt_1
	\cdots
	\int_0^{t_{n-1}} U(t_n)^* g_n(t_n) U(t_n)dt_n
	\right. \right. \\
	&\left. \left. \qquad \qquad \times
	\int_0^{t_n} U(\ta)^* g_{n+1}(\ta) U(\ta)d\ta
	\int_0^\ta U(\ta_1)^* \ga_f g_{n+2}(\ta_1) U(\ta_1)d\ta_1 \right] \right| \\
	&\quad \le \|U^\os W^0\|_{\FS^{10}_{\tx \to x}}
	\No{W^1(t)U(t)\int_0^t U(\ta)^* g_1^0(\ta) d\ta }_{\FS^{5}_\tx}
	\times \cdots \\
	&\qquad \cdots \times
	\No{g_n^1(t_n)U(t_n)\int_0^{t_n} U(\ta)^* g_{n+1}^0(\ta) d\ta }_{\FS^{5}_\tx} \\
	&\qquad \times
	\No{g_{n+1}^1(\ta)U(\ta) \int_0^\ta d\ta_1 \ga_f U(\ta_1) g_{n+2}^0(\ta_1)}_{\FS^{5/3}_\tx}
	\|g_{n+2}^1 U^\oc\|_{\FS^{10}_{x\to \tx}} \\
	&\quad \ls_f C_0^{n+2} \|W^0\|_{L^5_\tx} \|W^1\|_{L^5_\tx}
	\prod_{j=1}^n \|g_j\|_{L^2_t L^3_x}
	\|g_{n+1}^0\|_{L^5_\tx}
	\|g_{n+1}^1\|_{L^{20/7}_\tx}
	\|g_{n+2}^0\|_{L^{20/7}_\tx}
	\|g_{n+2}^1\|_{L^5_\tx}\\
	&\quad = C_0^{n+2}
	\|W\|_{L^{5/2}_\tx}
	\prod_{j=1}^{n} \|g_j\|_{L^2_t L^3_x}
	\|g_{n+1}\|_{L^{20/11}_\tx}
	\|g_{n+2}\|_{L^{20/11}_\tx}.
\end{align}

\subsubsection{With derivative}
Finally, we prove that
\begin{align} \label{eq:with deriv}
	\|\na \CN_1[g_1,\dots,g_{n+2}]\|_{L^{5/2}_\tx}
	\le C_0^{n+2} \prod_{j=1}^{n} \|\sd g_j\|_{L^2_t L^{3/2}_x} \|\sd g_{n+1}\|_{L^{20/9}_\tx} \|\sd g_{n+2}\|_{L^{20/9}_\tx}.
\end{align}
Let $W \in C_c^\I((0,\I)\times \BR^3)$.
Let $f_0 := |f|^\tw$ and $f = f_0 f_1$.
Then we have
\begin{align}
	&\abs{\int_0^\I \int_{\R^3} dt (\na W)(t,x) \CN_1 [g_1,\dots,g_{n+2}](t,x) dx dt } \\
	&\quad \le \left| \Tr \left[ \int_0^\I dt  U(t)^* W(t) U(t) \sdi
	\int_0^tdt_1 \sd U(t_1)^* g_1(t_1) U(t_1)
	\cdots
	\int_0^{t_1} dt_n U(t_n)^* g_n(t_n) U(t_n)
	\right. \right. \\
	&\qquad \left. \left. \qquad \times
	\int_0^{t_n} d\ta U(\ta)^* g_{n+1}(\ta) U(\ta)
	\int_0^{t_{n+1}} d\ta_1 U(\ta_1)^* \ga_f g_{n+2}(\ta_1) U(\ta_1)\na \right] \right| \\
	&\qquad + \left| \Tr \left[ \int_0^\I dt \sdi U(t)^* W(t) U(t) 
	\na \int_0^tdt_1 U(t_1)^* g_1(t_1) U(t_1)
	\cdots
	\int_0^{t_1} dt_n U(t_n)^* g_n(t_n) U(t_n)
	\right. \right. \\
	&\left. \left. \qquad \qquad \times
	\int_0^{t_n} d\ta U(\ta)^* g_{n+1}(\ta) U(\ta)
	\int_0^{t_{n+1}} d\ta_1 U(\ta_1)^* \ga_f g_{n+2}(\ta_1) U(\ta_1)\sd \right] \right|=:A + B.
\end{align}
Since we have
\begin{align}
	&\|\rh(U(t)_\st \ga_0)\|_{L^{5/2}_\tx} \ls \|\sd \ga_0\|_{\FS^{15/8}}, \\
	&\|\rh(U(t)_\st \ga_0)\|_{L^{20/11}_\tx} \ls \|\sd^{\tw} \ga_0\|_{\FS^{15/11}}, \\
	&\sd = \sdi - \na\cdot\na \sdi,
\end{align}
the duality argument and the usual Strchartz estimates implies
\begin{align}
	&A \le C_0^{n+2} \prod_{j=1}^{n} \|\sd g_j\|_{L^2_t L^{3/2}_x} \|\sd g_{n+1}\|_{L^{20/9}_\tx} \|\sd g_{n+2}\|_{L^{20/9}_\tx}, \\
	&B \le C_0^{n+2} \prod_{j=1}^{n} \|\sd g_j\|_{L^2_t L^{3/2}_x} \|\sd g_{n+1}\|_{L^{20/9}_\tx} \|\sd g_{n+2}\|_{L^{20/9}_\tx},
\end{align}
which yields the desired estimate.

\subsubsection{Conclusion}
By interpolating \eqref{eq:no deriv II} and \eqref{eq:with deriv}, we obtain
\begin{align} \label{eq:homo deriv}
	\||\na|^\tw \CN_1[g_1,\dots,g_{n+2}]\|_{L^2_t L^2_x} \le C_0^{n+2} \prod_{j=1}^{n+2} \|g_j\|_{L^2_t H^{1/2}_x}.
\end{align}
By \eqref{eq:no deriv} and \eqref{eq:homo deriv}, we get
\begin{align}
	\|\CN_1 [g_1,\dots,g_{n+2}]\|_{L^2_t H^{1/2}_x} \le C_0^{n+2} \prod_{j=1}^{n+2} \|g_j\|_{L^2_t H^{1/2}_x}.
\end{align}
Therefore, we have
\begin{align}
	\|N_2^1(V)\|_{L^2_t H^{1/2}_x} &\le \sum_{n=0}^\I \|\CN_n(V)\|_{L^2_t H^{1/2}_x}
	= \sum_{n=0}^\I \|\CN_n[V,\dots,V]\|_{L^2_t H^{1/2}_x} \\
	&\le \sum_{n=0}^\I C_0^{n+2} \|V\|_{L^2_t H^{1/2}_x}^{n+2}
	\ls \|V\|_{L^2_t H^{1/2}_x}^2
\end{align}
for sufficiently small $\|V\|_{L^2_t H^{1/2}_x}$.
We can prove \eqref{CC} in the same argument as in the proof of Lemmas \ref{Resol est}, \ref{almost commu} and Theorem \ref{Key}.

\subsection{Estimate for $N_2^2$}
By \eqref{eq:wave expansion}, we have
\begin{align}
	N_2^2(V) &= \sum_{n=0}^\I \int_0^t \int_0^\ta
	\rh \Big[U(t) \CW_V^{(n)}(t,\ta)U(\ta)^* V(\ta)U(\ta-\ta_1) V(\ta_1) \ga_f U(\ta_1-t) \Big] d\ta_1 d\ta \\
	&= \sum_{n=0}^\I \CM_n(V).
\end{align}
Define $\CM_n[g_1, \dots, g_{n+2}]$ in the same manner as $\CN_n[g_1,\dots, g_{n+2}]$.

\subsubsection{With no derivative I}
In the same way as we proved \eqref{eq:no deriv}, we have
\begin{align}\label{eq:with no deriv b I}
	\|\CM_n[g_1,\dots,g_{n+2}]\|_{L^2_t L^2_x} \le C_0^{n+2} \prod_{j=1}^{n+2} \|g_j\|_{L^2_t H^{1/2}_x}.
\end{align}

\subsubsection{With no derivative II}
In the same way as we proved \eqref{eq:no deriv II}, we get
\begin{align}\label{eq:with no deriv b II}
	\|\CM_n[g_1,\dots,g_{n+2}]\|_{L^{20/9}_\tx}
	\le C_0^{n+2} \prod_{j=1}^{n}
	\|g_j\|_{L^{5/2}_x}
	\|g_{n+1}\|_{L^{5/2}_\tx}
	\|g_{n+2}\|_{L^{20/11}_\tx}.
\end{align}

\subsubsection{With derivatives}
We will prove
\begin{align} \label{eq:with deriv b}
	\|\na \CM_n[g_1,\dots,g_{n+2}]\|_{L^{20/11}_\tx}
	\le C_0^{n+2} \prod_{j=1}^{n}
	\|\sd g_j\|_{L^{5/3}_\tx}
	\|\sd g_{n+1}\|_{L^{5/3}_\tx}
	\|\sd g_{n+2}\|_{L^{20/9}_\tx}.
\end{align}
First, note that we can prove the following lemma by the standard argument with the Whitney decomposition.
\begin{lemma} \label{lem:MCK}
	Let $X_1$, $X_2$ and $X$ be Banach spaces.
	Assume that bilinear maps $B,B_<:L^{p_1}_t(\BR_+,X_1) \times L^{p_2}_t(\BR_+,X_2) \to L^p_t(\BR_+,X)$ are defined by
	\begin{align}
		&B(g_1,g_2):= \int_0^\I dt_1 K_1(t,t_1)g_1(t_1) \int_0^\I dt_2 K_2(t_1,t_2)g_2(t_2), \\
		&B_<(g_1,g_2):= \int_0^t dt_1 K_1(t,t_1)g_1(t_1) \int_0^{t_1} dt_2 K_2(t_1,t_2)g_2(t_2).
	\end{align}
	If $p>p_1$ and $1/p_1+1/p_2 > 1$, then we have
	\begin{align}
		\|B(g_1,g_2)\|_{L^p_t X}
		\ls \|g_1\|_{L^{p_1}_t X_1}
		\|g_2\|_{L^{p_2}_t X_2}
		\Longrightarrow
		\|B_<(g_1,g_2)\|_{L^p_t X}
		\ls \|g_1\|_{L^{p_1}_t X_1}
		\|g_2\|_{L^{p_2}_t X_2}.
	\end{align}
\end{lemma}
Hence, it suffices to prove
\begin{align} \label{eq:M}
	\|\na\CM_n^*[g_1,\dots,g_{n+2}] \|_{L^{20/11}_\tx} \ls \|\sd g_1\|_{L^{5/3}_\tx}\cdots \|\sd g_{n+1}\|_{L^{5/3}_\tx} \|\sd g_{n+2}\|_{L^{20/9}_\tx},
\end{align}
where
\begin{align}
	&\CM^*_n[g_1, \dots, g_{n+2}](t)
	:= (-i)^{n} \int_0^\I dt_1 \cdots \int_0^\I dt_n \int_0^{\I} d\ta \int_0^\I d\ta_1 \\
	&\qquad \times \rh \Big[U(t) \CU[g_1](t_1) \cdots \CU[g_n](t_n) U(\ta)^*
	g_{n+1}(\ta) U(\ta-\ta_1)g_{n+2}(\ta_1) \ga_f U(\ta_1-t) \Big].
\end{align}
For simplicity, we only consider $\CM_0^*[g_1,g_2]$; we can deal with the general case in the completely same way.
Let $W \in C_c^\I((0,\I)\times \R^3)$.
Then we have
\begin{align}
	&\abs{\int_0^\I dt \int_{\R^3} dx W(t,x) \na \CM_0^*[g_1,g_2](t,x) } \\
	&\quad \le \No{\sd^{-\fh-\ep} \int_0^\I dt U(t)^* \na W(t) U(t)\sd^{-\fh} }_{\FS^{20/7}} \\
	&\qquad \times \No{\sd^{\fh} \int_0^\I dt_1 U(t_1)^* g_1(t_1) U(t_1) \int_0^\I dt_2 U(t_2)^* g_2(t_2)\ga_f U(t_2)^* \sd^{\fh+\ep} }_{\FS^{20/13}}
	= A B.
\end{align}
By \cite[Theorem 3.1]{2018CHP}, we have
\begin{align} \label{eq:3/2}
	\||\na|^{\frac{3}{2}}\rh(U(t)_\st \ga)\|_{L^2_\tx} \ls \|\sd^{1+\ep} \ga \sd\|_{\FS^2}.
\end{align}
By interpolating with
\begin{align}
	\||\na|^{\tw}\rh(U(t)_\st \ga)\|_{L^{5/3}_\tx} \ls \|\sd^\tw \ga \sd^\tw\|_{\FS^{5/4}},
\end{align}
we obtain
\begin{align}
	\||\na| \rh(U(t)_\st \ga)\|_{L^{20/11}_\tx} \ls \|\sd^{\fh+\ep} \ga \sd^\fh\|_{\FS^{20/13}}.
\end{align}
Note that the above $\ga$ does not need to be self-adjoint.
Hence we have
\begin{align}
	A \ls \|W\|_{L^{20/9}_\tx}.
\end{align}

Next, we estimate $B$.
On the one hand, we have
\begin{align}
	B_1&:= \No{\int_0^\I dt_1 U(t_1)^* g_1(t_1) U(t_1) \int_0^\I dt_2 U(t_2)^* g_2(t_2)\ga_f U(t_2)^* \sd^{\fh+\ep} }_{\FS^{20/13}} \\
	&\ls \No{\int_0^\I dt_1 U(t_1)^* g_1(t_1) U(t_1) \sdi \int_0^\I dt_2 U(t_2)^* (\na g_2)(t_2)\ga_f U(t_2)^* \sd^{\fh+\ep} }_{\FS^{20/13}} \\
	&\quad + \No{\int_0^\I dt_1 U(t_1)^* g_1(t_1) U(t_1) \sdi \int_0^\I dt_2 U(t_2)^* g_2(t_2)\na \ga_f U(t_2)^* \sd^{\fh+\ep} }_{\FS^{20/13}}.
\end{align}
Since Lemma \ref{asym} implies
\begin{align}
	&\|\rh(U(t)_\st \ga)\|_{L^{5/2}_t L^{60/29}_x} \ls \|\ga \sd^\fh \|_{\FS^{60/37}}, \\
	&\|\rh(U(t)_\st \ga)\|_{L^{20/11}_\tx} \ls \|\ga \sd^{\ft}\|_{\FS^{15/11}},
\end{align}
we have
\begin{align}
	&\No{\int_0^\I dt_1 U(t_1)^* g_1(t_1) U(t_1) \sdi \int_0^\I dt_2 U(t_2)^* g_2(t_2)\na \ga_f U(t_2)^* \sd^{\fh+\ep} }_{\FS^{20/13}} \\
	&\le \No{\int_0^\I U(t_1)^* g_1(t_1) U(t_1) dt_1 \sd^{-\fh}}_{\FS^{15/7}}
	\No{\sd^{-\ft} \int_0^\I U(t_2)^* g_2(t_2)\ga_f \na U(t_2)^* \sd^{\fh+\ep} dt_2 }_{\FS^{15/4}} \\
	&\ls \|g_1\|_{L^{5/3}_t L^{60/31}_x} \|g_2\|_{L^{20/9}_\tx} \|\lxr^{2} f(\xi)\|_{L^1_\xi \cap L^\I_\xi}
	\ls \|\sd g_1\|_{L^{5/3}_\tx} \|g_2\|_{L^{20/9}_\tx} \|\lxr^{2} f(\xi)\|_{L^1_\xi \cap L^\I_\xi}.
\end{align}
By interpolating
\begin{align}
	&\||\na| \rh(U(t)_\st \ga)\|_{L^2_\tx} \ls \|\sd^\tw \ga \sd^{1+\ep}\|_{\FS^2} \\
	&\||\na|^{\tw}\rh(U(t)_\st \ga)\|_{L^{5/3}_\tx} \ls \|\sd^\tw \ga \sd^\tw \|_{\FS^{5/4}},
\end{align}
we obtain
\begin{align}
	\||\na|^\frac{3}{4} \rh(U(t)_\st \ga)\|_{L^{20/11}_\tx} \ls \|\sd^\tw \ga \sd^{\frac{3}{4} + \ep}\|_{\FS^{20/13}}.
\end{align}
Hence, we conclude that
\begin{align}
	\No{\sd^{-\tw}\int_0^\I U(t_2)^* (\na g_2)(t_2)\ga_f U(t_2)^* \sd^{\fh+\ep} dt_2 }_{\FS^{20/7}}
	\ls \|\sd^{\ft} g_2\|_{L^{20/9}_\tx} \|\lxr^2 f(\xi)\|_{L^1_\xi \cap L^\I_\xi}.
\end{align}
Therefore, by using
\begin{align}
	\|\rh(U(t)_\st \ga)\|_{L^{5/2}_t L^{30/17}_x} \ls \|\ga \sd^\tw\|_{\FS^{10/7}},
\end{align}
we have
\begin{align}
	&\No{\int_0^\I dt_1 U(t_1)^* g_1(t_1) U(t_1) \sdi \int_0^\I dt_2 U(t_2)^* (\na g_2)(t_2)\ga_f U(t_2)^* \sd^{\fh+\ep} }_{\FS^{20/13}} \\
	&\ls \No{\int_0^\I dt_1 U(t_1)^* g_1(t_1) U(t_1) \sd^{-\tw}}_{\FS^{10/3}}
	\No{\sd^{-\tw}\int_0^\I dt_2 U(t_2)^* (\na g_2)(t_2) \ga_f U(t_2)^* \sd^{\fh+\ep} }_{\FS^{20/7}} \\
	&\ls \|g_1\|_{L^{5/3}_t L^{30/13}_x} \|\sd^{\ft} g_2\|_{L^{20/9}_\tx}  \|\lxr^2 f(\xi)\|_{L^1_\xi \cap L^\I_\xi} \\
	&\ls \|\sd g_1\|_{L^{5/3}_\tx} \|\sd^{\ft} g_2\|_{L^{20/9}_\tx}  \|\lxr^2 f(\xi)\|_{L^1_\xi \cap L^\I_\xi}.
\end{align}
From the above, we get
\begin{align}
	B_1 \ls \|\sd g_1\|_{L^{5/3}_\tx} \|\sd^{\ft} g_2\|_{L^{20/9}_\tx} \|\lxr^2 f(\xi)\|_{L^1_\xi \cap L^\I_\xi}.
\end{align}

On the other hand, we have
\begin{align}
	B_2&:= \No{\na \int_0^\I dt_1 U(t_1)^* g_1(t_1) U(t_1) \int_0^\I dt_2 U(t_2)^* g_2(t_2)\ga_f U(t_2)^* \sd^{\fh+\ep} }_{\FS^{20/13}} \\
	&\le \No{\int_0^\I dt_1 U(t_1)^* (\na g_1)(t_1) U(t_1) \int_0^\I dt_2 U(t_2)^* g_2(t_2)\ga_f U(t_2)^* \sd^{\fh+\ep} }_{\FS^{20/13}} \\
	&\quad + \No{\int_0^\I dt_1 U(t_1)^* g_1(t_1) U(t_1) \int_0^\I dt_2 U(t_2)^* (\na g_2)(t_2)\ga_f U(t_2)^* \sd^{\fh+\ep} }_{\FS^{20/13}} \\
	&\quad + \No{\int_0^\I dt_1 U(t_1)^* g_1(t_1) U(t_1) \int_0^\I dt_2 U(t_2)^* g_2(t_2) \na \ga_f U(t_2)^* \sd^{\fh+\ep} }_{\FS^{20/13}} \\
	&= C_1 + C_2 + C_3.
\end{align}
For $C_1$, we obtain
\begin{align}
	C_1 &\ls \No{\int_0^\I dt_1 U(t_1)^* (\na g_1)(t_1) U(t_1) \sdi \int_0^\I dt_2 U(t_2)^* (\na g_2)(t_2)\ga_f U(t_2)^* \sd^{\fh+\ep} }_{\FS^{20/13}} \\
	&\quad + \No{\int_0^\I dt_1 U(t_1)^* (\na g_1)(t_1) U(t_1) \sdi \int_0^\I dt_2 U(t_2)^* g_2(t_2)\na \ga_f U(t_2)^* \sd^{\fh+\ep}}_{\FS^{20/13}}
\end{align}
By using
\begin{align}
	&\|\rh(U(t)_\st \ga)\|_{L^{5/2}_x} \ls \|\ga \sd \|_{\FS^{15/8}}, \\
	&\|\rh(U(t)_\st \ga)\|_{L^{20/11}_t L^{30/19}_x} \ls \|\ga \|_{\FS^{60/49}},
\end{align}
we conclude that
\begin{align}
	&\No{\int_0^\I dt_1 U(t_1)^* (\na g_1)(t_1) U(t_1) \sdi \int_0^\I dt_2 U(t_2)^* g_2(t_2) \na \ga_f U(t_2)^* \sd^{\fh+\ep} }_{\FS^{20/13}} \\
	&\ls \No{\int_0^\I dt_1 U(t_1)^* (\na g_1)(t_1) U(t_1) \sdi}_{\FS^{15/7}}
	\No{\int_0^\I dt_2 U(t_2)^* g_2(t_2) \na \ga_f U(t_2)^* \sd^{\fh+\ep} }_{\FS^{60/11}} \\
	&\ls \|\sd g_1\|_{L^{5/3}_\tx} \|g_2\|_{L^{20/9}_t L^{30/11}_x} \|\lxr^2 f(\xi)\|_{L^1_\xi \cap L^\I_\xi}
	\ls \|\sd g_1\|_{L^{5/3}_\tx} \|\sd g_2\|_{L^{20/9}_\tx} \|\lxr^2 f(\xi)\|_{L^1_\xi \cap L^\I_\xi}.
\end{align}
By using
\begin{align}
	&\|\rh(U(t)_\st \ga)\|_{L^{5/2}_x} \ls \|\ga \sd \|_{\FS^{15/8}}, \\
	&\|\rh(U(t)_\st \ga)\|_{L^{20/11}_\tx} \ls \|\ga \sd^\ft\|_{\FS^{15/11}},
\end{align}
we get
\begin{align}
	&\No{\int_0^\I dt_1 U(t_1)^* (\na g_1)(t_1) U(t_1) \sdi \int_0^\I dt_2 U(t_2)^* (\na g_2)(t_2)\ga_f U(t_2)^* \sd^{\fh+\ep} }_{\FS^{20/13}} \\
	&\ls \|\sd g_1\|_{L^{5/3}_\tx} \|\sd g_2\|_{L^{20/9}_\tx}\|\lxr^2 f(\xi)\|_{L^1_\xi \cap L^\I_\xi}.
\end{align}
From the above, we have
\begin{align}
	C_1 &\ls \|\sd g_1\|_{L^{5/3}_\tx} \|\sd g_2\|_{L^{20/9}_\tx} \|\lxr^{2}f(\xi)\|_{L^1_\xi \cap L^\I_\xi}.
\end{align}
In the same way, we obtain
\begin{align}
	C_3 \le \|\sd g_1\|_{L^{5/3}_\tx} \|\sd g_2\|_{L^{20/9}_\tx} \|\lxr^{3}f(\xi)\|_{L^1_\xi \cap L^\I_\xi}
\end{align}
For $C_2$, we have
\begin{align}
	C_2 &\le
	\No{ 
		\int_0^\I dt_1 U(t_1)^* g_1(t_1) U(t_1) \sd^{-\tw}
	}_{\FS^{10/3}} \\
	&\qquad \times \No{\int_0^\I dt_2 \sd^{\tw} U(t_2)^* (\na g_2)(t_2)\ga_f U(t_2)^* \sd^{\fh+\ep} }_{\FS^{20/7}}.
\end{align}
Since
\begin{align}
	\|\rh(U(t)_\st \ga)\|_{L^{5/2}_t L^{30/17}_x} \ls \|\ga \sd^\tw\|_{\FS^{10/7}},
\end{align}
we have
\begin{align}
	\No{ 
		\int_0^\I dt_1 U(t_1)^* g_1(t_1) U(t_1) \sd^{-\tw}
	}_{\FS^{10/3}}
	\ls \|g_1\|_{L^{5/3}_t L^{30/13}_x} \ls \|\sd g_1\|_{L^{5/3}_\tx}.
\end{align}
Note that
\begin{align}
	&\No{\sd^\tw \int_0^\I dt_2 U(t_2)^* (\na g_2)(t_2)\ga_f U(t_2)^* \sd^{\fh+\ep} }_{\FS^{20/7}} \\
	&\ls \No{\sd^{-\tw} \int_0^\I dt_2 U(t_2)^* (|\na|^2 g_2)(t_2)\ga_f U(t_2)^* \sd^{\fh+\ep} }_{\FS^{20/7}} \\
	&\quad + \No{\sd^{-\tw} \int_0^\I dt_2 U(t_2)^* (\na g_2)(t_2)\na \ga_f U(t_2)^* \sd^{\fh+\ep} }_{\FS^{20/7}}.
\end{align}
By interpolating
\begin{align}
	&\||\na| \rh(U(t)_\st \ga)\|_{L^2_\tx} \ls \|\sd^\tw \ga \sd^{1+\ep}\|_{\FS^2} \\
	&\||\na|^{\tw}\rh(U(t)_\st \ga)\|_{L^{5/3}_\tx} \ls \|\sd^\tw \ga \sd^\tw \|_{\FS^{5/4}},
\end{align}
we obtain
\begin{align}
	\||\na|^\frac{3}{4} \rh(U(t)_\st \ga)\|_{L^{20/11}_\tx} \ls \|\sd^\tw \ga \sd^{\frac{3}{4} + \ep}\|_{\FS^{20/13}}.
\end{align}
Hence, we have
\begin{align}
	&\No{\sd^{-\tw} \int_0^\I dt_2 U(t_2)^* (|\na|^2 g_2)(t_2)\ga_f U(t_2)^* \sd^{\fh+\ep} }_{\FS^{20/7}} \\
	&\ls \||\na|^{\frac{5}{4}}g_2\|_{L^{20/9}_\tx} \|\lxr^2 f(\xi)\|_{L^1_\xi \cap L^\I_\xi}
	\ls \|\sd^{\frac{5}{4}} g_2\|_{L^{20/9}_\tx} \|\lxr^2 f(\xi)\|_{L^1_\xi \cap L^\I_\xi}.
\end{align}
From the above, we conclude that
\begin{align}
	C_2 \ls \|\sd g_1\|_{L^{5/3}_\tx} \|\sd^{\frac{5}{4}} g_2\|_{L^{20/9}_\tx} \|\lxr^3 f(\xi)\|_{L^1_\xi \cap L^\I_\xi}.
\end{align}
Finally, interpolating
\begin{align}
	B_1&=\No{\int_0^\I dt_1 U(t_1)^* g_1(t_1) U(t_1) \int_0^\I dt_2 U(t_2)^* g_2(t_2)\ga_f U(t_2)^* \sd^{\fh+\ep} }_{\FS^{20/13}} \\
	&\ls \|\sd g_1\|_{L^{5/3}_\tx} \|\sd^\ft g_2\|_{L^{20/9}_\tx} \|\lxr^2 f(\xi)\|_{L^1_\xi \cap L^\I_\xi}, \\
	B_2&=\No{\na \int_0^\I dt_1 U(t_1)^* g_1(t_1) U(t_1) \int_0^\I dt_2 U(t_2)^* g_2(t_2)\ga_f U(t_2)^* \sd^{\fh+\ep} }_{\FS^{20/13}} \\
	&\ls \|\sd g_1\|_{L^{5/3}_\tx} \|\sd^{\frac{5}{4}} g_2\|_{L^{20/9}_\tx} \|\lxr^3 f(\xi)\|_{L^1_\xi \cap L^\I_\xi}.
\end{align}
we obtain
\begin{align}
	B&=\No{\sd^{\fh}\int_0^\I dt_1 U(t_1)^* g_1(t_1) U(t_1) \int_0^\I dt_2 U(t_2)^* g_2(t_2)\ga_f U(t_2)^* \sd^{\fh+\ep} }_{\FS^{20/13}} \\
	&\ls \|\sd g_1\|_{L^{5/3}_\tx} \|\sd g_2\|_{L^{20/9}_\tx} \|\lxr^{3} f(\xi)\|_{L^1_\xi \cap L^\I_\xi}.
\end{align} 

\subsubsection{Conclusion}
By interpolating \eqref{eq:with no deriv b II} and \eqref{eq:with deriv b}, we obtain
\begin{align} \label{eq:homo deriv 2}
	\||\na|^\tw \CM_1[g_1,\dots,g_{n+2}]\|_{L^2_t L^2_x} \le C_0^{n+2} \prod_{j=1}^{n+2} \|g_j\|_{L^2_t H^{1/2}_x}.
\end{align}
By \eqref{eq:with no deriv b I} and \eqref{eq:homo deriv 2}, we get
\begin{align}
	\|\CM_1 [g_1,\dots,g_{n+2}]\|_{L^2_t H^{1/2}_x} \le C_0^{n+2} \prod_{j=1}^{n+2} \|g_j\|_{L^2_t H^{1/2}_x}.
\end{align}
Therefore, we have
\begin{align}
	\|N_2^2(V)\|_{L^2_t H^{1/2}_x} &\le \sum_{n=0}^\I \|\CM_n(V)\|_{L^2_t H^{1/2}_x}
	= \sum_{n=0}^\I \|\CM_n[V,\dots,V]\|_{L^2_t H^{1/2}_x} \\
	&\le \sum_{n=0}^\I C_0^{n+2} \|V\|_{L^2_t H^{1/2}_x}^{n+2}
	\ls \|V\|_{L^2_t H^{1/2}_x}^2
\end{align}
for sufficiently small $\|V\|_{L^2_t H^{1/2}_x}$.
We can prove \eqref{CC} in the same argument as in the proof of Lemmas \ref{Resol est}, \ref{almost commu} and Theorem \ref{Key}.

\section{Proof of Lemma \ref{reduction}}
\begin{proof}
Let  
\begin{align}
	Q(t) &= U_V(t)_\st Q_0 - i\int_0^t U_V(t,\ta)_\st [V(\ta),\ga_f] d\ta \\
	&=: Q_1(t) + Q_2(t).
\end{align}
First we consider $Q_1(t)$.
We have
\begin{align}
 \sd^\tw Q_1(t) \sd^\tw = (\sd^\tw U_V(t) \sd^{-\tw})_\st (\sd^\tw Q_0 \sd^\tw).
\end{align}
It is easliy proven by \eqref{Duh1} that $\sd^\tw U_V(t) \sd^{-\tw}$ is strongly continuous.
Therefore, $Q_1(t) \in C(\BR, \CH^\tw)$.
Next we consider $Q_2(t)$.
We obtain
\begin{align}
 Q_2(t+h) - Q_2(t) &= -i \int_t^{t+h} U_V(t+h,\ta)[V(\ta),\ga_f]U_V(\ta,t+h) d\ta \\
 &\quad  -i \int_0^t (U_V(t+h,\ta)-U_V(t,\ta))[V(\ta),\ga_f]U_V(\ta,t+h) d\ta \\
 &\quad  -i \int_0^t U_V(t,\ta)[V(\ta),\ga_f](U_V(\ta,t+h)-U_V(\ta,t)) d\ta \\
 &=: A + B + C.
\end{align}
Lemma \ref{almost commu} and Kato--Seiler--Simon inequality (\cite{1975SS}; see also \cite[Theorem 4.1]{2005Simon}) imply
\begin{align}
 &\|\sd^\tw A \sd^\tw\|_{\FS^2} \\
 &\le \int_t^{t+h} \|\sd^\tw U_V(t+h,\ta) \sd^{-\tw}\|_{\CB(L^2_x)}
                   \|\sd^\tw [V(\ta),\ga_f] \sd^\tw\|_{\FS^2}
                   \|\sd^{-\tw} U_V(\ta,t) \sd^\tw\|_{\CB(L^2_x)} d\ta \\
 &\le \ph(\|V\|_{L^2_t H^{1/2}_x}) \|\lxr f(\xi)\|_{L^2_\xi} \int_t^{t+h} \|V(\ta)\|_{H^{1/2}_x}d\ta
 \to 0 \mbox{ as } h \to 0.
\end{align}
By the similar argument and the fact that $\sd^\tw U_V(t) \sd^{-\tw}$ is strongly continuous, we have
\begin{align}
 \|\sd^\tw B \sd^\tw\|_{\FS^2} \to 0 \mbox{ as } h \to 0, \quad 
 \|\sd^\tw C \sd^\tw\|_{\FS^2} \to 0 \mbox{ as } h \to 0,
\end{align}
which yields $Q_2(t) \in C(\BR, \CH^\tw)$.
Therefore, we obtain $Q(t) \in C(\BR, \CH^\tw)$.

Next, we prove the scattering.
First we consider $Q_1(t) \in C(\BR,\CH^\tw)$.
We have
\begin{align}
 &\|U(-t)Q_1(t)U(t)-U(-s)Q_1(s)U(s)\|_{\FS^3} \\
 &\quad \le \|U(-t)U_V(t)-U(-s)U_V(s)\|_{\CB} \|Q_0\|_{\FS^3}
            +\|Q_0\|_{\FS^3}  \|U_V(t)^*U(t)-U_V(s)^*U(s)\|_{\CB}.
\end{align}
Since
\begin{align}
	\|(U(-t)U_V(t)-U(-s)U_V(s))u_0\|_{L^2_x}
	&= \No{ \int_s^t U(\ta)^* V(\ta) U_V(\ta) u_0 d\ta }_{L^2_x} \\
	&\ls \|V(\ta) U_V(\ta)u_0\|_{L^1_\ta([s,t],L^2_x)} \\
	&\le \|V\|_{L^2_\ta([s,t], L^3_x)} \|U_V(\ta)u_0\|_{L^2_\ta L^6_x} \\
	&\le \ph(\|V\|_{L^2_\ta L^3_x}) \|V\|_{L^2_\ta([s,t], L^3_x)} \|u_0\|_{L^2_x},
\end{align}
we obtain
\begin{align} \label{sc}
	\|U(-t)U_V(t)-U(-s)U_V(s)\|_\CB \le \ph(\|V\|_{L^2_\ta L^3_x}) \|V\|_{L^2_\ta([s,t], L^3_x)} \to 0
	\mbox{ as } t,s \to \I.
\end{align}
Therefore, we get
\begin{align}
	\|U(-t)Q_1(t)U(t)-U(-s)Q_1(s)U(s)\|_{\FS^3} \to 0 \mbox{ as } t,s \to \I.
\end{align}
Next we consider $Q_2(t) \in C(\BR,\CH^\tw)$.
We have
\begin{align}
 &\|U(-t)Q_2(t)U(t)-U(-s)Q_2(s)U(s)\|_{\FS^3} \\
 &\quad \le \|U(-t)U_V(t)-U(-s)U_V(s)\|_\CB
        \No{\int_0^t U_V(\ta)^* [V(\ta),\ga_f]U_V(\ta) d\ta}_{\FS^3} \\
 &\qquad + \No{ \int_s^t U_V(\ta)^*[V(\ta),\ga_f] U_V(\ta) d\ta }_{\FS^3} \\
 &\qquad + \No{\int_0^s U_V(\ta)^* [V(\ta),\ga_f]U_V(\ta) d\ta}_{\FS^3} 
           \|U(-t)U_V(t)-U(-s)U_V(s)\|_\CB.
\end{align}
Let $V_0 := |V|^{\frac{2}{5}}$ and $V = V_0 V_1$.
Then we have
\begin{align}
	\No{\int_I U_V(t)^*V(t)\ga_f U_V(t) dt}_{\FS^3}
	&\le \No{\int_I dt U_V(t)^* V_0(t) }_{\FS^{10}_{(\tx)\to x}} \|V_1 \ga_f U_V(t)\|_{\FS^{30/7}_{x \to(t,x)}} \\
	&\le \ph(\|V\|_{L^2_t L^3_x}) \|\lxr^\tw f(\xi)\|_{L^2_\xi} \|V\|_{L^2_t(I,L^2_x)}.
\end{align}
Therefore, we obtain
\begin{align}
	\No{ \int_s^t U_V(\ta)^*[V(\ta),\ga_f] U_V(\ta) d\ta }_{\FS^3}
	&\le \ph(\|V\|_{L^2_\ta L^3_x}) \|\lxr^\tw f(\xi)\|_{L^2_\xi} \|V\|_{L^2_\ta([s,t],L^2_x)} \nonumber \\
	&\to 0 \mbox{ as } t,s \to \I. \label{f sc}
\end{align}
It follows from \eqref{sc} and \eqref{f sc} that
\begin{align}
	\|U(-t)Q_2(t)U(t)-U(-s)Q_2(s)U(s)\|_{\FS^3} \to 0 \mbox{ as } t,s \to \I.
\end{align}
\end{proof}

\section*{Acknowledgments}
The author would like to thank Professor Kenji Nakanishi at Kyoto University for many comments and discussions. Professor Julien Sabin and Mr. Antoine Borie at the University of Rennes carefully read the first version of this paper uploaded to arXiv. Moreover, they gave the author many useful comments and spent much time discussing it. In particular, they pointed out that the argument in Sect.5 of the first version needed to be modified. The author deeply appreciates their contributions.

This work was supported by JST, the establishment of university fellowships towards the creation of science technology innovation, Grant Number JPMJFS2123.


\begin{thebibliography}{99}
%\bibitem{1990A}
%Araki, H.:
%On an inequality of Lieb and Thirring.
%Lett. Math. Phys. {\bf 19}, 167–-170 (1990).

\bibitem{2019BHLNS}
Bez, N., Hong, Y., Lee, S., Nakamura, S., Sawano, Y.:
On the Strichartz estimates for orthonormal systems of initial data with regularity.
Adv. Math. {\bf 354} (2019).

\bibitem{2003BS}
Birman, M.S., Solomyak, M.:
Double operator integrals in a Hilbert space.
Integral Equations Operator Theory {\bf 47} (2003), no.2, 131--168.


\bibitem{1974BDF}
Bove, A., Da Prato, G., Fano, G.:
An existence proof for the Hartree-Fock time-dependent problem with bounded two-body interaction.
Comm. Math. Phys.
{\bf 37} (1974), 183--191.

\bibitem{1976BDF}
Bove, A., Da Prato, G., Fano, G.:
On the Hartree-Fock time-dependent problem.
Comm. Math. Phys.
{\bf 49} (1976), 25--33.

\bibitem{1976C}
Chadam, J.M.:
The time-dependent Hartree-Fock equations with Coulomb two-body interaction.
\newblock Comm. Math. Phys. 
{\bf 46} (1976), 99--104.

\bibitem{2001CK}
Christ, M., Kiselev, A.:
Maximal functions associated to filtrations.
J. Funct. Anal. {\bf 179} (2001), no.2, 409--425.

\bibitem{2017CHP}
Chen, T., Hong, Y., Pavlovi\'{c}, N.: Global well-posedness of the NLS system for infinitely many fermions. Arch. Ration. Mech. Anal. {\bf 224} (2017), no.1, 91--123.


\bibitem{2018CHP}
Chen, T., Hong, Y., Pavlovi\'{c}, N.:
On the scattering problem for infinitely many fermions in dimensions $d \geq 3$ positive temperature.
Ann. Inst. H. Poincar\'{e} Anal. Non Lin\'{e}aire {\bf 35} (2018), no. 2, 393--416.

\bibitem{2020CS}
Collot, C., de Suzzoni, A.-S.:
Stability of equilibria for a Hartree equation for random fields.
J. Math. Pures Appl. (9) {\bf 137} (2020), 70--100.

\bibitem{2022CS}
Collot, C., de Suzzoni, A.-S.:
Stability of steady states for Hartree and Schrödinger equations for infinitely many particles.
Ann. H. Lebesgue {\bf 5} (2022), 429--490.





\bibitem{2021D}
Dong, X.:
The Hartree equation with a constant magnetic field: well-posedness theory.
Lett. Math. Phys. {\bf 111} (2021), no.4, Paper No. 101.


\bibitem{2014FLLS} Frank, R.L., Lewin, M., Lieb, E.H., Seiringer, R.:
Strichartz inequality for orthonormal functions,
J. Eur. Math. Soc. {\bf 16} (2014), no. 7, 1507--1526.


\bibitem{2017FS} Frank, R.L., Sabin, J.:
Restriction theorems for orthonormal functions, Strichartz inequalities, and uniform Sobolev estimates.
Amer. J. Math. {\bf 139} (2017), no. 6, 1649--1691.


\bibitem{1970Go-Kre}
Gohberg, I.C., Kre\u{\i}n, M.G.:
Theory and applications of Volterra operators in Hilbert space.
Transl. Math. Monogr., Vol. {\bf 24},
American Mathematical Society, Providence, RI, (1970).


\bibitem{2023H} Hadama, S.:
Asymptotic stability of a wide class of steady states for the Hartree equation for random fields.
(2023) arXiv:2303.02907.


\bibitem{1998KT}
Keel, M., Tao, T.:
Endpoint Strichartz estimates.
Amer. J. Math. {\bf 120} (1998), no.5, 955--980.

%\bibitem{1976LT}
%Lieb, E.H., Thirring, W.E.:
%Inequalities for the Moments of the Eigenvalues of the Schrödinger Hamiltonian and their Relation to Sobolev Inequalities.
%Studies in Mathematical Physics. pp. 269--303.
%Princeton University Press, Princeton (1976).


\bibitem{2015LS}
Lewin, M., Sabin, J.: The Hartree equation for infinitely many particles I. Well-posedness theory. Comm. Math. Phys. {\bf 334} (2015), no.1, 117--170.

\bibitem{2014LS}
Lewin, M., Sabin, J.: The Hartree equation for infinitely many particles II: Dispersion and scattering in 2D. Anal. PDE {\bf 7} (2014), no. 6, 1339--1363.

\bibitem{2020LS}
Lewin, M., Sabin, J.:
The Hartree and Vlasov equations at positive density.
Comm. Partial Differential Equations {\bf 45} (2020), no. 12, 1702--1754.

\bibitem{2021PS}
Pusateri, F., Sigal, I.M.:
Long-time behaviour of time-dependent density functional theory.
Arch. Ration. Mech. Anal. {\bf 241} (2021), no. 1, 447--473.

\bibitem{1975SS}
Seiler, E., Simon, B.:
Bounds in the Yukawa2 quantum field theory: upper bound on the pressure, Hamiltonian bound and linear lower bound.
Comm. Math. Phys. {\bf 45} (1975), no. 2, 99--114.

\bibitem{2005Simon}
Simon, B.:
{\em Trace ideals and their applications, Second edition.}
Mathematical Surveys and Monographs. 
{\bf 120}. American Mathematical Society, (2005).

\bibitem{2015S}
de Suzzoni, A.-S.:
An equation on random variables and systems of fermions.
(2015) arXiv:1507.06180.

\bibitem{1992Z}
Zagatti, S.:
The Cauchy problem for Hartree-Fock time-dependent equations.
Ann. Inst. H. Poincaré Phys. Théor. {\bf 56} (1992), 357--374.
\end{thebibliography}
\end{document}